\documentclass{article}
\def\regularformat{%
 \addtolength{\oddsidemargin}{-.875in}%
 \addtolength{\evensidemargin}{-.875in}%
 \addtolength{\textwidth}{1.75in}%
 \addtolength{\topmargin}{-.875in}%
 \addtolength{\textheight}{1.75in}%
}
\regularformat
\usepackage{amsthm}
\usepackage{amsmath}
\usepackage{amssymb}
\usepackage{graphicx}
\usepackage{float}
\usepackage[font=small,labelsep=none]{caption}
\usepackage{xcolor}
\usepackage{url}
\newtheorem{theorem}{Theorem}[section]
\newtheorem{lemma}[theorem]{Lemma}
\newtheorem{proposition}[theorem]{Proposition}
\newtheorem{corollary}[theorem]{Corollary}

\newtheorem{claim}{Claim}
 \long\def\ignore#1{} % Alternative way to comment out large blocks!
 
 \let\reallabel\label
 \def\labelshow#1{\reallabel{#1}[#1]}
 \def\showlabels{\let\label\labelshow}
 \let\whenfig\noignore

\def\sem(#1;#2){\Sigma(#1\,|\,#2)}
\def\apm(#1;#2){\Theta(#1\,|\,#2)}
\let\bv\beta
\let\pf\pi
\def\deg{\textup{deg}}
\def\hf{\frac12}
\let\att=z
\def\set#1{\{#1\}}
\def\cart{\mskip1.5mu\square\mskip1mu} % cartesian prod; \, is 3mu

\begin{document}

\title{\bf Hamiltonicity of planar graphs with a forbidden minor}

\author{%
  M. N. Ellingham%
	\thanks{Supported by National Security Agency grant
H98230-13-1-0233 and Simons Foundation awards 245715 and 429625.  The
United States Government is authorized to reproduce and distribute
reprints notwithstanding any copyright notation herein.}%
    % note stuff to add footnote mark 1
	\\
  Department of Mathematics, 1326 Stevenson Center\\
  Vanderbilt University, Nashville, Tennessee 37212, U.S.A.\\
  \texttt{mark.ellingham@vanderbilt.edu}%
 \and
  Emily A. Marshall%
	\thanks{Supported by National Security Agency grant
H98230-13-1-0233 and Simons Foundation award 245715.}%
	\\
  Department of Mathematics, 303 Lockett Hall\\
  Louisiana State University, Baton Rouge, Louisiana 70803, U.S.A.\\
  \texttt{emarshall@lsu.edu}
 \and
  % silly kludge to prevent authors being combined horizontally
  \hbox to 0.8\hsize{\hss
    Kenta Ozeki%
	\footnote{Supported in part by JSPS KAKENHI Grant Number
25871053 %
	and by a Grant for Basic Science Research Projects from The %
	Sumitomo Foundation.}%
	\hss}\\
National Institute of Informatics,\\
2-1-2 Hitotsubashi, Chiyoda-ku, Tokyo 101-8430, Japan \\
and \\
  JST, ERATO, Kawarabayashi Large Graph Project, Japan\\
  \texttt{ozeki@nii.ac.jp}%
 \and
  Shoichi Tsuchiya%
	%\footnote{?? Grant info Tsuchiya.}%
	\\
  School of Network and Information, Senshu University,\\
	2-1-1 Higashimita, Tama-ku, Kawasaki-shi, Kanagawa, 214-8580,
Japan\\ 
	\texttt{s.tsuchiya@isc.senshu-u.ac.jp}%
 %  Department of Mathematics, Keio University\\
 %	3-14-1 Hiyoshi, Kohoku-ku, Yokohama-shi, Kanagawa 223-8522, Japan\\
 %	\texttt{wco.liew6.te@gmail.com}%
}

\date{October 20, 2016}

\maketitle

\begin{abstract}
 Tutte showed that $4$-connected planar graphs are Hamiltonian, but it
is well known that $3$-connected planar graphs need not be Hamiltonian.
 We show that $K_{2,5}$-minor-free $3$-connected planar graphs are
Hamiltonian.  This does not extend to $K_{2,5}$-minor-free $3$-connected
graphs in general, as shown by the Petersen graph, and does not extend
to $K_{2,6}$-minor-free $3$-connected planar graphs, as we show by an
infinite family of examples. \end{abstract}

\section{Introduction}

 All graphs in this paper are finite and simple (no loops or multiple
edges).

 Whitney \cite{whitney} showed that every $4$-connected plane
triangulation is Hamiltonian, and Tutte \cite{tutte} extended this to
every $4$-connected planar graph.  Tutte's result has been strengthened
in various ways; see for example
\cite{chiba,HS09,ozeki,Sa97,TY94,thomassen}.

 If we relax the connectivity condition, it is not true that all $2$- or
$3$-connected planar graphs are Hamiltonian.  The smallest $2$-connected
planar graph that is not Hamiltonian is $K_{2,3}$.  The smallest
$3$-connected planar graph that is not Hamiltonian is the so-called
Herschel graph, with $11$ vertices and $18$ edges.  It was known to
Coxeter in 1948 \cite[p.~8]{Co48}, but a proof that it is smallest
relies on later work by Barnette and Jucovi\v{c} \cite{BJ70} and
Dillencourt \cite{Di96}.  If we restrict to triangulations, the smallest
$2$- or $3$-connected planar triangulation that is not Hamiltonian is a
triangulation obtained by adding $9$ edges to the Herschel graph.  It
was known to C. N. Reynolds (in dual form) in 1931, as reported by
Whitney \cite[Fig.~9]{whitney}.  Again, the proof that this is smallest
relies on \cite{BJ70} and \cite{Di96}.  This triangulation was also
presented much later by Goldner and Harary \cite{GH75}, so it is
sometimes called the Goldner-Harary graph.

 It is therefore reasonable to ask what conditions can be imposed on a
$2$- or $3$-connected planar graph to make it Hamiltonian.  The main
direction in which positive results have been obtained is to restrict
the types of $2$- or $3$-cuts in the graph.  Dillencourt \cite[Theorem
4.1]{dillencourt} showed that a near-triangulation (a $2$-connected
plane graph with all faces bounded by triangles except perhaps the
outer face) with no separating triangles and certain restrictions on
chords of the outer cycle is Hamiltonian.  Sanders \cite[Theorem
2]{sanders} extended this to a larger class of graphs.  Jackson and Yu
\cite[Theorem 4.2]{jackson} showed that a plane triangulation is
Hamiltonian if each `piece' defined by decomposing along separating
triangles connects to at most three other pieces.  Our results explore a
different kind of condition, based on excluding a complete bipartite
minor.

 %  Dillencourt, Theorem 4.1: near-triangulation is Hamiltonian if (i) no
 % separating triangles, (ii) each face of subgraph induced by boundary has
 % at most three chordal edges.
 %  Sanders (ham cyc cert planar), Theorem 2 generalizes Dillencourt to
 % non-near-triangulations (condition a bit technical, but again involves
 % placement of 3-cuts).
 %  Jackson and Yu, Theorem 4.2: triangulation is Hamiltonian if each
 % "piece" defined by decomposing on separating triangles connects to at
 % most three other pieces.  (Reynolds has decomposition tree with two
 % vertices of degree $4$.)

 Excluded complete bipartite minors have been used previously in more
general settings to prove results involving concepts related to the
existence of a Hamilton cycle, such as toughness, circumference, or the
existence of spanning trees of bounded degree; see for example
 \cite{CEKMO11,% Toughness and $K_{a,t}$-minor 
 	chen,% circum and $K_{3,t}$-minor
	OO12}. % sp trees bounded degree Ota and Ozeki
 We are interested in graphs that have no $K_{2,t}$ minor, for some
$t$.  Some general results are known for such graphs, including
a rough structure theorem \cite{ding2}, upper bounds on the number of edges
\cite{chud, myers}, and a lower bound on circumference \cite{chen}.

 For $2$-connected graphs, a $K_{2,3}$-minor-free graph is either
outerplanar or $K_4$, and is therefore both planar and Hamiltonian.
 However, the authors \cite{K24char} recently characterized all
$K_{2,4}$-minor-free graphs, and there are many $K_{2,4}$-minor-free
$2$-connected planar graphs that are not Hamiltonian.
 For $3$-connected graphs, the $K_{2,4}$-minor-free ones belong to a
small number of small graphs, some of which are nonplanar, or a sparse
infinite family of planar graphs; all are Hamiltonian.
 There are $K_{2,5}$-minor-free $3$-connected nonplanar graphs that are
not Hamiltonian, such as the Petersen graph, but in this paper we show
that all $K_{2,5}$-minor-free $3$-connected planar graphs are
Hamiltonian.
 We also show that this cannot be extended to $K_{2,6}$-minor-free
graphs, by constructing an infinite family of $K_{2,6}$-minor-free
$3$-connected planar graphs that are not Hamiltonian.

 The number $g(n)$ of nonisomorphic
$K_{2,5}$-minor-free $3$-connected planar graphs on $n$ vertices
 grows at least exponentially (for $n \ge 10$ with $n$ even this is not hard to show
using the family of graphs obtained by adding an optional diagonal chord
across each quadrilateral face of a prism $C_{n/2} \cart K_2$).  Some
computed values of $g(n)$ are as follows.

 \begin{center}
 \begin{tabular}{c|cccccc}
 $n$ & $7$ & $8$ & $9$ & $10$ & $11$ & $12$ \\ \hline
 $g(n)$ & $31$ & $194$ & $918$ & $3\,278$ &
	$8\,346$ & $18\,154$ \\
 \end{tabular}
 \end{center}

 \noindent
 The exponential growth of $g(n)$ contrasts with the growth of the
number of nonisomorphic $3$-connected $K_{2,4}$-minor-free graphs
(planar or nonplanar), which is only linear \cite{K24char}.
 Thus, our results apply to a sizable class of graphs.

 In Section 2 we provide necessary definitions and preliminary
results.
 The main result, Theorem~\ref{thm:main}, that $K_{2,5}$-minor-free
$3$-connected planar graphs are Hamiltonian, is proved in Section 3. In
Section 4 we discuss $K_{2,6}$-minor-free $3$-connected planar graphs.

 % Not using material on cycle lengths, cubic with face restrictions

 % Instead of looking for Hamilton cycles in graphs, a related idea is to
 % show that graphs have cycles of at least some minimum length. Chen and
 % Yu \cite{chen} proved that $3$-connected planar $n$-vertex graphs have
 % cycles of length at least $cn^{\text{log}_32}$.
 %
 % Restricting the vertex degrees so that every vertex has degree three
 % (cubic graphs) is not sufficient to guarantee Hamiltonicity for
 % $3$-connected planar graphs \cite{tutte46}. Barnette and Goodey
 % conjectured that every $3$-connected cubic planar graph with all face
 % degrees at most six is Hamiltonian (see \cite{malkevitch} and
 % \cite{goodey77}). Barnette also conjectured that every $3$-connected
 % cubic bipartite planar graph is Hamiltonian (see \cite{grunbaum70}).

%%%%%%%%%%%%%%%%%%%%%%%%%%%%%%%%%%%%%%%%%%%%%%%%%%%%%%%%%%%%%%%%%%%%%%%%%%%%%%%%%%%%%%%%%%%%%%%%%%%%%%%%%%%%%%%%%%%%%%%%%%%%%%%%%%%%%%%%%%%%%%%%%%%%%%%%%%%%%%%%%%%%%
\section{Definitions and Preliminary Results}

 An edge, vertex, or set of $k$ vertices whose deletion increases the
number of components of a
graph is a \textit{cutedge}, \textit{cutvertex}, or \textit{$k$-cut},
respectively.
 The subgraph of $G$ induced by $S \subseteq V(G)$ is denoted by $G[S]$.
 If $P$ is a path and $x, y \in V(P)$ then $P[x,y]$ represents the
subpath of $P$ between $x$ and $y$.

 \subsection{Minors and models}

 A graph $H$ is a \textit{minor} of a graph $G$ if $H$ is isomorphic to a
graph formed from $G$ by contracting and deleting edges of $G$ and
deleting vertices of $G$.
 A graph is \textit{$H$-minor-free} if it does not have $H$ as a
minor.
 Another way to think of a minor $H$ is in terms of a function
$\bv$ mapping each $u \in V(H)$ to $\bv(u) \subseteq V(G)$, the
\textit{branch set} of $u$, such that
 (a) $\bv(u) \cap \bv(u') = \emptyset$ if $u \ne u'$;
 (b) $G[\,\bv(u)\,]$ is connected for each $u$; and
 (c) if $uu' \in E(H)$ then there is at least one edge between $\bv(u)$
and $\bv(u')$ in $G$.
 We call $\bv$ a \textit{model}, or more specifically an
\textit{edge-based model}, of $H$ in $G$.
 More generally, we may replace condition (c) by the existence of a
function $\pf$ mapping each $e=uu' \in E(H)$ to a path $\pf(uu')$ in
$G$, such that
 ($\hbox{c}_1$) $\pf(uu')$ starts in $\bv(u)$ and ends in $\bv(u')$;
 ($\hbox{c}_2$) no internal vertex of $\pf(uu')$ belongs to any
$\bv(u'')$ (even if $u'' = u$ or $u'$); and
 ($\hbox{c}_3$) $\pf(e)$ and $\pf(e')$ are internally disjoint if $e \ne e'$.
 We call $(\bv,\pf)$ a \textit{path-based model} of $H$ in $G$.

 Now we discuss $K_{2,t}$ minors in particular.
 We assume that $V(K_{2,t}) = \{a_1, a_2, b_1, b_2, \ldots, b_t\}$ and
$E(K_{2,t}) = \{ a_i b_j \,|\, 1 \le i \le 2,\, 1 \le j \le t\}$.

 Edge-based models are convenient for proving nonexistence of a minor.
In fact, for $K_{2,t}$ minors we can use an even more restrictive model.
 Consider an edge-based model $\bv_0$ of $K_{2,t}$, where $b_1$ and its
incident edges correspond to a path $v_1v_2\ldots v_k$, $k \geq 3$, with
$v_1 \in \bv_0(a_1)$, $v_k \in \bv_0(a_2)$, and $v_i \in \bv_0(b_1)$ for
$2 \leq i \leq k-1$.  Define $\bv_1(b_1)=\{v_2\}$,
$\bv_1(a_2)=\bv_0(a_2) \cup \{v_3,\ldots ,v_{k-1}\}$, and $\bv_1(u) =
\bv_0(u)$ for all other $u$.  Then $\bv_1$ is also an edge-based model
of a $K_{2,t}$ minor, and $|\bv_1(b_1)|=1$.  Applying the same procedure
to $b_2, b_3, \ldots, b_t$ in turn, we obtain an edge-based model
$\bv=\bv_t$ with $|\bv(b_j)| = 1$ for $1 \le j \le t$.
 Such a $\bv$ is a \textit{standard model} of $K_{2,t}$, and we denote
it by either $\sem(R_1, R_2; s_1, s_2, \ldots, s_t)$ or $\sem(R_1, R_2;
S)$ where $R_i = \bv(a_i)$ for $1 \le i \le 2$, $\{s_j\} = \bv(b_j)$ for
$1 \le j \le t$, and $S = \{s_1, s_2, \ldots, s_t\}$.
 Thus, if $G$ has a $K_{2,t}$ minor then it has a standard model of
$K_{2,t}$.

 On the other hand, path-based models are useful for proving
existence of a minor.  Moreover, in many situations where we find a
$K_{2,t}$ minor it would be tedious to give an exact description of the
minor.  So, we say that
 $\apm(R_1, R_2; S_1, S_2, \ldots, S_t)$ is an \textit{approximate
(path-based) model} of $K_{2,t}$ if
there exists a path-based model $(\bv, \pf)$ such that
$R_i \subseteq \bv(a_i)$ for $1 \le i \le 2$ and $S_j \subseteq
\bv(b_j)$ for $1 \le j \le t$.
 For convenience, we also allow each $R_i$ to be a subgraph, not just a
set of vertices, with $V(R_i) \subseteq \bv(a_i)$, and similarly for
each $S_j$.
 Informally, we will specify enough of each branch set to make the
existence of the minor clear; using subgraphs rather than vertex sets
sometimes helps to clarify why a branch set induces a connected
subgraph.
 We use this notation even when we actually have an exact description of
a minor.
 For brevity, we just say that `$\apm(R_1, R_2; S_1, S_2, \ldots, S_t)$
is a $K_{2,t}$ minor'.
 In our figures the sets or subgraphs $R_i$ are enclosed by dotted
curves, and the sets or subgraphs $S_j$ (usually just single vertices)
are indicated by triangles.

 We also need one type of rooted minor.
 We say there is a \textit{$K_{2,t}$ minor rooted at $R_1$ and $R_2$} if
there is a path-based model $(\bv, \pf)$ (or, equivalently, edge-based
model $\bv$, or even standard model $\bv$) with
 $R_1 \subseteq \bv(a_1)$ and  $R_2 \subseteq \bv(a_2)$.
 Again we extend this to allow $R_1$ and $R_2$ to be subgraphs, not just
sets of vertices.

 \subsection{Path-outerplanar graphs}

 A graph is \textit{outerplanar} if it has a plane embedding in which
all vertices are on the outer face.
 We use a characterization that we proved elsewhere of graphs without
rooted $K_{2,2}$ minors in terms of special outerplanar graphs.
 (Along different lines, Demasi \cite[Lemma 2.2.2]{demasi} provided a
description of graphs with no $K_{2,2}$ minor rooting all four vertices,
in terms of disjoint paths.)
 Our characterization uses the following definitions.
 Given $x, y \in V(G)$, an \textit{$xy$-outerplane embedding} of a graph
$G$ is an embedding in a closed disk $D$ such that a Hamilton $xy$-path
$P$ of $G$ is contained in the boundary of $D$; $P$ is called the
\textit{outer path}. A graph is \textit{$xy$-outerplanar}, or
\textit{path-outerplanar}, if it has an $xy$-outerplane embedding.
 A graph $G$ is a \textit{block} if it is connected and has no
cutvertex; a block is either $2$-connected, $K_2$, or $K_1$.

 \begin{lemma}[\cite{K24char}]
 \label{lem:rooted22}
 Suppose $x, y \in V(G)$ where $x \ne y$ and $G'=G+xy$ is a block (which
holds, in particular, if $G$ has a Hamilton $xy$-path). Then $G$ has no
$K_{2,2}$ minor rooted at $x$ and $y$ if and only if $G$ is
xy-outerplanar.
 \end{lemma}

 The following results on Hamilton paths in outerplanar and
$xy$-outerplanar graphs will be useful.

\begin{lemma}
Let $G$ be a $2$-connected outerplanar graph. Let $x \in V(G)$ and let
$xy$ be an edge on the outer cycle $Z$ of $G$. Then for some vertex $t$
with $\deg_G(t)=2$, there exists a Hamilton path $xy\ldots t$ in $G$.
\label{lem:P0}
\end{lemma}

\begin{proof}
 Fix a forward direction on $Z$ so that $y$ follows $x$.
 Denote by $v_1Zv_2$ the forward path from $v_1$ to $v_2$ on $Z$.
 Proceed by induction on $|V(G)|$. In the base case, $G=K_3$ and the
result is clear. Assume the lemma holds for all graphs with at most
$n-1$ vertices and $|V(G)|=n \geq 4$.
 Let $w \neq y$ be the other
neighbor of $x$ on $Z$. If $\deg_G(w)=2$, then we take $t=w$ and
$xZw$ is a desired Hamilton path in $G$. Otherwise let $v$ be a neighbor
of $w$ such that $vw \notin E(Z)$ (possibly $v=y$). Let $G'$ be the
subgraph of $G$ induced by $vZw$; $G'$ is a $2$-connected $vw$-outerplanar
graph with $|V(G')|\leq n-1$. By the inductive hypothesis, there exists
a Hamilton path $Q=vw\ldots t$ in $G'$ where
$\deg_{G'}(t)=\deg_G(t)=2$. Then $xZv \cup Q$ is the desired
path in $G$.
\end{proof}

\begin{corollary}
 Let $G$ be an $xy$-outerplanar graph with $x \ne y$.  Then there exists
a Hamilton path $x \ldots t$ in $G-y$, where $t=x$ if $|V(G)|=2$, and $t$
is some vertex with $\deg_G(t)=2$ otherwise.
\label{cor:P1}
\end{corollary}

 \begin{proof}
 If $|V(G)| = 2$ this is clear, so suppose that $|V(G)| \ge 3$.
 Then $G+xy$ is a $2$-connected outerplanar graph, so by Lemma
\ref{lem:P0} it has a Hamilton path $yx \ldots t$ ending at a vertex $t$
of degree $2$.  Now $P-y$ is the required path.
 \end{proof}

 \subsection{Connectivity and reducibility}

 The following observation will be useful.

 \begin{lemma}\label{lem:insidecycle}
 Suppose $G$ is a $2$-connected plane graph and $C$ is a cycle in $G$. 
Then the subgraph of $G$ consisting of $C$ and all edges and vertices inside
$C$ is $2$-connected.
 \end{lemma}

 \begin{proof}
 Any cutvertex in the subgraph would also be a cutvertex of $G$.
 \end{proof}

 The following results will allow us to simplify the situations that we
have to deal with in the proof of Theorem~\ref{thm:main}.

\begin{theorem}[Halin, {\cite[Theorem 7.2]{Halin}}] Let $G$ be a
$3$-connected graph with $|V(G)| \geq 5$. Then for every $v \in V(G)$ with
$\deg(v) = 3$, there is an edge $e$ incident with $v$ such that
$G/e$ is $3$-connected.
\label{thm:Halin}
\end{theorem}

A \textit{$k$-separation} in a graph $G$ is a pair $(H,K)$ of
edge-disjoint subgraphs of $G$ with $G=H\cup K$, $|V(H) \cap V(K)|=k$,
$V(H)-V(K) \neq \emptyset$, and $V(K)-V(H) \neq \emptyset$.

 %******** Old version.  K' is automatically nonempty, needn't say that
 % \begin{lemma} Let $G$ be a 3-connected graph and suppose $(H,K)$ is a
 % 3-separation in $G$ with $V(H) \cap V(G) = \{x,y,z\}$. Suppose
 % $K'=K-V(H)$ is nonempty and connected, each of $x,y$, and $z$ is
 % adjacent to a vertex of $K'$, and $H$ is $2$-connected. Let $G'$ be the
 % graph formed from $G$ by contracting $K'$ to a single vertex. Then $G'$
 % is 3-connected. Furthermore, for every cycle $Z'$ in $G'$ there is a
 % cycle $Z$ in $G$ with $|V(Z)| \geq |V(Z')|$.
 % \label{lem:sep}
 % \end{lemma}

 \begin{lemma}
 Let $G$ be a $3$-connected graph and suppose $(H,K)$ is a
$3$-separation in $G$ with $V(H) \cap V(G) = \{x,y,z\}$. Suppose
$K'=K-V(H)$ is connected and $H$ is $2$-connected. Let $G'$ be the graph formed
from $G$ by contracting $K'$ to a single vertex. Then $G'$ is
$3$-connected.
 \label{lem:sep}
 \end{lemma}

\begin{proof} Let $v$ be the vertex in $G'$ formed from contracting
$K'$.
 Since $G$ is $3$-connected, $xv, yv, zv \in E(G')$.  We claim that
every pair of vertices in $G'$ has three vertex-disjoint paths between
them. By Menger's Theorem, it will follow that $G'$ is $3$-connected. We
consider five different types of pairs of vertices.

First, suppose $w_1,w_2 \in V(H)-\{x,y,z\}$; there are three internally
disjoint paths from $w_1$ to $w_2$ in $G$: $P_1$, $P_2$, and $P_3$. If
$V(P_i) \cap V(K') = \emptyset$ for $i=1,2,3$, then $P_1$, $P_2$, and
$P_3$ are the desired paths in $G'$. If $V(P_i) \cap V(K') \neq
\emptyset$ for some $i$, then $|V(P_i) \cap \{x,y,z\}| \geq 2$ since
$\{x,y,z\}$ separates $K'$ from $H$. Thus $V(P_i) \cap V(K') \neq
\emptyset$ for at most one $i$. Suppose $V(P_1) \cap V(K') \neq
\emptyset$. Then all vertices of $V(P_1) \cap V(K')$ are in a single
subpath of $P_1$ which we replace by $v$ to form a new path $P_1'$. The
paths $P_1'$, $P_2$, and $P_3$ are the desired paths in $G'$.

 %*** OLD ARGUMENT NOT CORRECT IF $w_1$ adjacent to $x$. ***
 %  Second, consider $w_1 \in V(H)-\{x,y,z\}$ and $w_2 \in \{x,y,z\}$, say
 % $w_2=x$. If there are not three internally disjoint paths between $w_1$
 % and $x$ in $G'$, then there is a $2$-cut $\{u_1,u_2\}$ that separates
 % $w_1$ and $x$. Since there is no $2$-cut in $G$, one of $u_1$ and $u_2$
 % must be $v$, say $u_2=v$. The vertex $u_1$, however, is a cutvertex in
 % $H$ which is a contradiction since $H$ is $2$-connected. Hence there are
 % three internally disjoint paths between $w_1$ and $x$.

 Second, consider $w_1 \in V(H)-\{x,y,z\}$ and $w_2 \in \{x,y,z\}$, say
$w_2=x$. If there are not three internally disjoint paths between $w_1$
and $x$ in $G'$, then $w_1$ and $x$ are separated either by a $2$-cut
$\{u_1, u_2\}$ (if $w_1 x \notin E(G)$) or by $w_1x$ and some vertex
$u_1$ (if $w_1 x \in E(G)$).
 Since $w_1$ and $x$ are not separated by a $2$-cut or by an edge and a
vertex in $G$, we may assume that $u_1 = v$.
 But then $u_2$ is a cutvertex in $H$ or $w_1 x$ is a cutedge in $H$,
which is a contradiction since $H$ is $2$-connected. Hence there are
three internally disjoint paths between $w_1$ and $x$.

Third, consider $w_1,w_2 \in \{x,y,z\}$, say $w_1=x$ and $w_2=y$. Because
$H$ is $2$-connected, there are two internally disjoint paths $P_1$ and
$P_2$ from $x$ to $y$ in $H$. Take $P_3=xvy$. Then $P_1, P_2,$ and $P_3$
are the desired paths in $G'$.

Fourth, consider $w_1 \in V(H)-\{x,y,z\}$ and $v$. For any $w_2 \in
V(K')$, there are three internally disjoint paths $P_1$, $P_2$, and
$P_3$ from $w_2$ to $w_1$ in $G$. Without loss of generality, say $x \in
V(P_1)$, $y \in V(P_2)$, and $z \in V(P_3)$. Form $P_1'$ from $P_1$ by
replacing $P_1[w_2, x]$ with $vx$, form $P_2'$ from $P_2$ by replacing
$P_2[w_2, y]$ with $vy$, and form $P_3'$ from $P_3$ by replacing
$P_3[w_2, z]$ with $vz$. The paths $P_1'$, $P_2'$, and $P_3'$ are the
desired paths in $G'$.

 Finally, consider $w_1 \in \{x,y,z\}$, say $w_1=x$, and $v$. By a
consequence of Menger's Theorem, since $H$ is $2$-connected there are
two internally disjoint paths from $\{y,z\}$ to $x$ in $H$, say
$P_1=y\ldots x$ and $P_2=z\ldots x$. Then $P_1'=vy \cup P_1$, $P_2'=vz \cup
P_2$, and $P_3=vx$ are the desired paths in $G'$.
 \end{proof}

 Lemma \ref{lem:sep} is false without the hypothesis that $H$ is
$2$-connected: then we could have $V(H) = \{w,x,y,z\}$ and $E(H) = \{wx,
wy, wz\}$, in which case $G'$ would be isomorphic to $K_{2,3}$, which is
not $3$-connected.

 Now we use the results above to set up a framework that will help to
simplify the graph in our main proof.
 Suppose $G$ is a $3$-connected graph, and $C$ is a cycle in $G$.
 We say that \textit{$G$ is $C$-reducible to a graph $G'$} provided
 (a) $G'$ is obtained from $G$ by contracting edges of
$G$ with at most one end on $C$ and/or deleting edges in $E(G)-E(C)$,
 (b) $G'$ is $3$-connected, and
 (c) for every cycle $Z'$ in $G'$ there is a cycle $Z$ in $G$ with
$|V(Z)| \ge |V(Z')|$.
 By (a), $C$ is still a cycle in $G'$.  From this, we see that
$C$-reducibility is transitive.  Also by (a), $G'$ is a minor of $G$.

 \begin{lemma}\label{lem:redcomp}
 Suppose $C$ is a cycle in a $3$-connected graph $G$.
 If $B$ is a component of $G-V(C)$ with exactly three neighbors on $C$
then $G$ is $C$-reducible to $G/E(B)$, in which $B$ becomes a degree $3$
vertex.
 \end{lemma}

 \begin{proof}
 Let $G_0 = G-V(B)$.  If $G_0$ is not $2$-connected, then there
is a cutvertex $u$. Now $u \notin V(C)$ and $V(C)$ must be entirely in
one component of $G_0-u$. Since the neighbors of $B$ are
all on $C$, vertices of $B$ are only adjacent to vertices on one
side of the cut.  Hence $u$ is also a cutvertex in $G$, which is a
contradiction. Thus, $G_0$ is $2$-connected.
 Consider $G'=G/E(B)$.  Clearly (a) holds, and (b) follows from
Lemma~\ref{lem:sep}.

 Let $a_1, a_2, a_3$ be the neighbors of $B$ on $C$, and let $b$ be the
vertex of $G'$ corresponding to $B$.
 Let $Z'$ be a cycle in $G'$. If $b \notin V(Z')$, then $Z=Z'$ is
also a cycle in $G$. If $b \in V(Z)$ then $Z'$ uses a path $a_iba_j$.
 Form a cycle $Z$ in $G$ from $Z'$ by replacing $a_iba_j$ by a path from
$a_i$ to $a_j$ through $B$.  Clearly $|V(Z)| \ge |V(Z')|$, so (c) holds.
 \end{proof}

 \begin{lemma}\label{lem:reddeg3}
 Suppose $C$ is a cycle in a $3$-connected graph $G$.
 If $b \in V(G)-V(C)$ has degree $3$ then there is an edge
$bc$ so that $G$ is $C$-reducible to $G/bc$.
 \end{lemma}

 \begin{proof}
 By Theorem \ref{thm:Halin} there is an edge $bc$ such that $G'=G/bc$
is $3$-connected.  Clearly (a) and (b) hold for $G'$; we must show (c).
 Let $a_1, a_2$ and $c$ be the neighbors of $b$ in $G$.
 Call the vertex that results from the contraction $z$. Suppose
$Z'$ is a cycle in $G'$.  If $a_1z, a_2z \notin E(Z')$, then take
$Z=Z'$. If $|\{a_1z, a_2z\} \cap E(Z')|=1$, say $a_1z \in E(Z')$, form
$Z$ from $Z'$ by replacing $a_1z$ with the path $a_1bc$. If $a_1z, a_2z
\in E(Z')$, form $Z$ from $Z'$ by replacing the subpath $a_1za_2$ with
$a_1ba_2$. In all cases, $Z$ is a cycle in $G$ with $|V(Z)| \geq
|V(Z')|$, so (c) holds.
 \end{proof}

 \begin{lemma}\label{lem:rededge}
 Suppose $C$ is a cycle in a $3$-connected graph $G$. Suppose that $a_1
a_2 \in E(G)-E(C)$, and there are three internally disjoint
$a_1a_2$-paths in $G-a_1 a_2$.  Then $G$ is $C$-reducible to $G-a_1a_2$.

 In particular, $G$ is $C$-reducible to $G-a_1a_2$ if $a_1$ and $a_2$
are neighbors on $C$ of a component of $G-V(C)$ and $a_1 a_2 \in
E(G)-E(C)$.
 \end{lemma}

 \begin{proof}
 Clearly (a) and (c) hold for $G'=G-a_1a_2$; we must show (b).
 Since $G$ is $3$-connected, $G'$ is $2$-connected, and if $G'$ has a
$2$-cut then $a_1$ and $a_2$ must be in different components, which
cannot happen because of the three internally disjoint $a_1a_2$-paths.

 If $a_1$ and $a_2$ are neighbors of a component $B$ of $G-V(C)$ then
there are three internally disjoint $a_1a_2$-paths in $G-a_1a_2$,
namely the two paths between $a_1$ and $a_2$ in $C$, and a path from
$a_1$ to $a_2$ through $B$.
 \end{proof}

%%%%%%%%%%%%%%%%%%%%%%%%%%%%%%%%%%%%%%%%%%%%%%%%%%%%%%%%%%%%%%%%%%%%%%%%%%%%%%%%%%%%%%%%%%%%%%%%%%%%%%%%%%%%%%%%%%%%%%%%%%%%%%%%%%%%%%%%%%%%%%%%%%%%%%%%%%%%%%%%%%%%%%%
\section{Main Result}

We are now ready to prove the main result.

\begin{theorem}  Let $G$ be a $3$-connected planar $K_{2,5}$-minor-free
graph. Then $G$ is Hamiltonian.
\label{thm:main}
\end{theorem}

Theorem~\ref{thm:main} is proved by assuming $G$ is not Hamiltonian,
taking a longest cycle $C$ in $G$ and finding a contradiction with
either a longer cycle or a $K_{2,5}$ minor.

\begin{proof}
 Assume that $G$ is not Hamiltonian and assume $G$ is represented as a
plane graph.
 Let $H$ and $J$ be two subgraphs of $G$.  Let $R_0$ be the outside face
of $J$ (an open set), $R_1$ the boundary of $R_0$, and $R_2 =
\mathbb{R}^2-R_0-R_1$.  We say $H$ is \textit{outside $J$} if as subsets
of the plane we have $H
\subseteq R_0 \cup R_1$, and \textit{inside $J$} if $H \subseteq R_1 \cup
R_2$.

 Let $C$ be a longest non-Hamilton cycle in $G$. A \textit{longer cycle}
means a cycle longer than $C$.  Fix a forward direction on $C$, which we
assume is clockwise. Denote by $x^+$ the vertex directly after the
vertex $x$ on $C$ and by $x^-$ the vertex directly before $x$. Define
$C[x,y]$ to be the forward subpath of $C$ from $x$ to $y$ which includes
$x$ and $y$. If $x=y$ then $C[x,y]=\{x\}$. Define
$C(x,y)=C[x,y]-\{x,y\}$, $C(x,y]=C[x,y]-x$, and $C[x,y)=C[x,y]-y$.
Define $[x,y]$ to be $V(C[x,y])$ and $G[x,y]$ to be the induced subgraph
$G[\,[x,y]\,]$; also define $(x,y)$, $G(x,y)$, etc. similarly.
 We say a vertex $z$ is \textit{between} $x$ and $y$ if $z \in (x,y)$.

 Let $D$ be a component of $G-V(C)$ with the most neighbors on $C$.
 We fix $D$ in our arguments, and assume that $D$ is inside $C$.
 Let $u_0,u_1,\ldots ,u_{k-1}$ be the neighbors of $D$ along $C$ in
forward order. Because $G$ is $3$-connected, $k \geq 3$.
 For any distinct $u_i$ and $u_j$ there is at least one path from $u_i$
to $u_j$ through $D$; we use $u_i D u_j$ to denote such a path.
 The sets $U_i = (u_i, u_{i+1})$ (subscripts interpreted modulo $k$)
are called \textit{sectors}.
 If $U_i = \emptyset$ for some $i$, then there is a longer cycle:
replace $C[u_i,u_{i+1}]$ with $u_iDu_{i+1}$. Thus, $U_i \neq \emptyset$ for all $i$.

 A \textit{jump $x-y$} is an $xy$-path where $x \ne y$, $x, y \in V(C)$,
and no edge or internal vertex of the path belongs to $C$ or $D$.
 If $S, T \subseteq V(C)$ then a \textit{jump from $S$ to $T$} or
\textit{$S-T$ jump} is a jump $x-y$ with $x \in S$, $y \in T$; if $S=T$
we say this is a \textit{jump on $S$}.
 If $S$ is a set of consecutive vertices on $C$ then a \textit{jump out
of $S$} is a jump $x-y$ where $x \in S$, $y \notin S$, and $y$ is not
adjacent in $C$ to a vertex of $S$.
 Whenever $v, w \in V(C)$ are not equal and not consecutive on $C$ and
$(v,w)$ contains no neighbor of $D$ there is at least one jump out of
$(v,w)=[v^+,w^-]$, because $\{v,w\}$ is not a $2$-cut.

 A jump out of a sector $U_i$ is a \textit{sector jump}; since every $U_i$
is nonempty, there is a sector
jump out of every sector.
 A jump is an \textit{inside} or \textit{outside jump} if
it is respectively inside or outside $C$.
 An inside jump must have both ends in $[u_i,u_{i+1}]$ for some $i$.
 Thus, all sector jumps are outside jumps.

 If there is a jump $u_i^+-u_j^+$, then $C[u_j^+,u_i] \cup u_iDu_j \cup
C[u_i^+,u_j] \cup u_i^+-u_j^+$ is a longer cycle. Denote such a
longer cycle as $L(u_i^+-u_j^+)$. If there is a jump $u_i^--u_j^-$, then
there is a symmetric longer cycle denoted $L(u_i^--u_j^-)$. Call such
cycles \textit{standard longer cycles}.
 \whenfig{Figure \ref{fig:P18} shows $L(u_1^-, u_2^-)$ when $k=4$.}

 If $x, y \in V(C)$, $x \ne y$, $W \subseteq G-V(C)-V(D)$, and
$G[\,[x,y] \cup W\,]$ contains a $K_{2,2}$ minor rooted at $x$ and $y$,
then we say there is a \textit{$K_{2,2}$ minor along $[x,y]$}.  If there
is no such minor then for any $[x',y'] \subseteq [x,y]$ with $x' \ne y'$
there is no $K_{2,2}$ minor rooted at $x'$ and $y'$ in $G[x',y']$. 
Thus, $G[x',y']$ is $x'y'$-outerplanar by Lemma \ref{lem:rooted22} and
we may apply
 Corollary \ref{cor:P1} to $G[x',y']$.

 Suppose $a,b,c,d$ with $c \ne b$, $a \ne d$ appear in that order along
$C$. Let $W_1, W_2 \subseteq G-V(C)-V(D)$ with $W_1 \cap W_2 =
\emptyset$.
 If there is a $K_{2,2}$ minor in $G[\,[a,d] \cup W_1\,]$
 rooted at $[a,b]$ and $[c,d]$,
 represented as $\sem(R_1, R_2; s_1, s_2)$,
 and a $K_{2,2}$ minor in $G[\,[c,b] \cup W_2\,]$
 rooted at $[a,b]$ and $[c,d]$,
 represented as $\sem(R_1', R_2'; s_1', s_2')$,
and there exist $u_i \in [a,b]$ and $u_j \in [c,d]$, then
there is a $K_{2,5}$ minor $\apm(R_1 \cup R_1', R_2 \cup R_2'; s_1, s_2,
s_1', s_2', D)$ in $G$. Denote such a minor by
$M([a,b],[c,d])$.
 \whenfig{An example is shown in Figure~\ref{fig:genminor}.}

\begin{figure}
\centering
\begin{minipage}[b]{.33\linewidth}
\centering \scalebox{.95}{\includegraphics{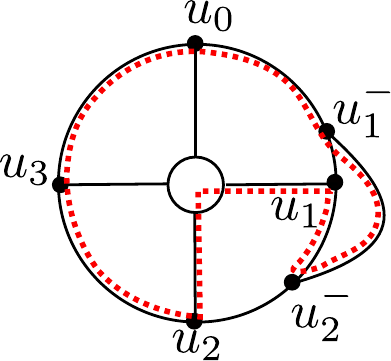}}
\captionof{figure}{\label{fig:P18}}
\end{minipage}% Need % so no space added which could force a new line
\begin{minipage}[b]{.33\linewidth}
\centering \scalebox{.95}{\includegraphics{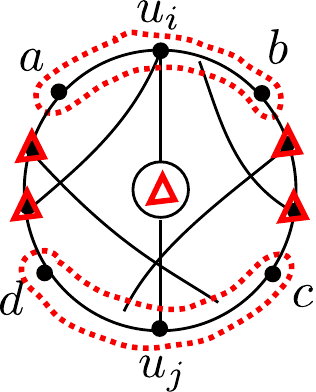}}
\captionof{figure}{\label{fig:genminor}}
\end{minipage}%
\end{figure}

 For $x \in V(C)$, define $\sigma(x) \in \{0, \hf, 1, 1\hf,
\ldots, k-\hf\}$ by $\sigma(u_i)=i$, and $\sigma(x)=i+\hf$ if
$x \in U_i$.
 Define the \textit{length} of a jump $x-y$ as
$\text{min}\{|\sigma(x)-\sigma(y)|,k-|\sigma(x)-\sigma(y)|\}$.
 A sector jump has length at least $1$.

\begin{claim} For every jump $x-y$ of length greater than $1$,
there is a sector jump $x_1-y_1$ of length $1$ with $x_1,y_1 \in [x,y]$
and another sector jump $x_2-y_2$ of length $1$ with $x_2,y_2 \in
[y,x]$.
\label{clm:jump1}
\end{claim}

 For any jump $u-v$, define the \textit{linear length} as
$|\sigma(u)-\sigma(v)|$.
 We claim that for any jump (not necessarily a sector jump) $x'-y'$ of
linear length $\ell'>1$ with $\sigma(x') < \sigma(y')$, there is a
sector jump $x''-y''$ of linear length less than $\ell'$ with $x'', y''
\in [x', y']$.  The jump $x'-y'$ must be outside $C$, and there is a
sector $U_j \subset (x',y')$. Let $x''-y''$ be any jump
out of $U_j$; then $\sigma(x') < \sigma(x'') < \sigma(y')$. If $x''-y''$
does not contain an interior vertex of $x'-y'$, then by
planarity $x''-y''$ has linear length less than $\ell'$.
 If $x''-y''$ contains an interior vertex of $x'-y'$, then we have jumps
$x''-x'$ and $x''-y'$ with linear length less than $\ell'$, at least one
of which is a sector jump.
 We may repeat this process until we reach a sector jump $x^*-y^*$ with
$x^*,y^* \in [x',y']$ of linear length $1$, and hence also length $1$.

If we relabel $u_0,u_1,\ldots,u_{k-1}$ keeping the same cyclic order so
that $x \in \{u_0\} \cup U_0$ and repeatedly apply the previous
paragraph beginning with the jump $x-y$, we obtain the required jump
$x_1-y_1$. Similarly, relabeling so that $y \in \{u_0\} \cup U_0$ yields
the jump $x_2-y_2$. This completes the proof of Claim~\ref{clm:jump1}.

 \medskip

 \begin{claim} $k = 3$.
 \label{clm:1}
 \end{claim}

 Assume that $k \geq 4$.
 Suppose there is a component $D'$ of $G-V(C)$ with neighbors in
three consecutive sectors, say $\att_1 \in U_0$, $\att_2 \in U_1$, and
$\att_3 \in U_2$ ($D'$ may also have neighbors in other sectors). Then
since $k \geq 4$, $\att_1-\att_3$ is a jump of length greater than $1$.
Therefore by Claim~\ref{clm:jump1}, there is a sector jump $x-y$ of
length $1$ with $u_i \in [x,y] \subseteq [\att_3, \att_1]$.  At most one
of $x \in U_2$, $y \in U_0$ is true; we may assume that $y \notin U_0$.
 Then there is a $K_{2,5}$ minor
 $\apm(D \cup \set{u_1}, D' \cup [\att_3, x] \cup x-y;
	u_0, \att_1, \att_2, u_2, u_i)$%
  \whenfig{ as shown in Figure~\ref{fig:P4}}.
 This minor applies even if $x-y$ intersects $D'$.

 Now suppose there is a component $D'$ of $G-V(C)$ with neighbors in three
sectors that are not consecutive (this requires $k \ge 5$; again $D'$
may also have neighbors in other sectors).
 We may assume that these are $\att_1 \in U_h$, $\att_2 \in U_i$,
$\att_3 \in U_j$ in order along $C$, where $U_h, U_i$ may be consecutive
but $U_i, U_j$ and $U_j, U_h$ are not.
 Then there is a $K_{2,5}$-minor
 $\apm(D \cup \set{u_{i+1}}, D' \cup \set{\att_1, \att_3};
	u_h, u_i, \att_2, u_j, u_{j+1})$.
 \whenfig{An example with $(h,i,j)=(k-1, 0, 2)$ is shown in
Figure~\ref{fig:P1-2}.
 }

 Hence, every component of $G-V(C)$ other than $D$ has neighbors in at
most two sectors. Therefore, a sector jump of length $1$, from $U_{i-1}$
to $U_i$, cannot intersect any sector jump with an end in $U_j$, $j
\notin \set{i-1, i}$, which includes all sector jumps of length at least
$2$.

\begin{figure}
\centering
\begin{minipage}[b]{.33\linewidth}
\centering \scalebox{.95}{\includegraphics{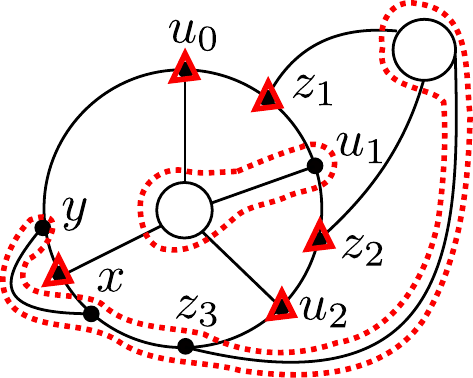}}
\captionof{figure}{\label{fig:P4}}
\end{minipage}%
\begin{minipage}[b]{.33\linewidth}
\centering \scalebox{.9}{\includegraphics{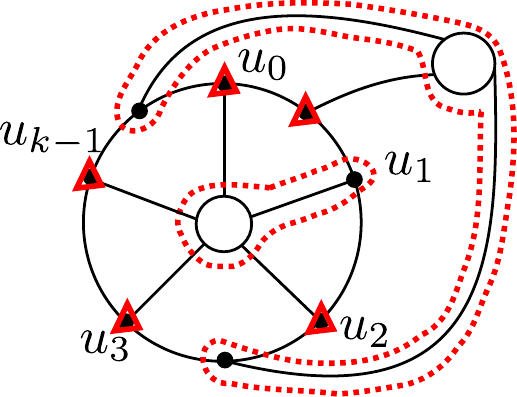}}
\captionof{figure}{\label{fig:P1-2}}
\end{minipage}%
\end{figure}

 From Claim \ref{clm:jump1} it follows that there are at least two
distinct pairs of sectors with jumps of length $1$ between them.
 Suppose there are three distinct pairs of sectors with jumps of length
$1$ between them, say $x_1-y_1$, $x_2-y_2$ and $x_3-y_3$
in order along $C$, where $u_g \in (x_1,y_1)$, $u_h \in (x_2,y_2)$ and $u_i
\in (x_3,y_3)$.  Since $k \ge 4$, we may assume there is some $u_j \in
(y_3,x_1)$.
 Then there is a $K_{2,5}$ minor
 $\apm(D \cup \set{u_j}, [y_1,x_2] \cup x_2-y_2 \cup [y_2,x_3];
	x_1, u_g, u_h, u_i, y_3)$.
 \whenfig{An example with $(g,h,i)=(0,1,2)$ is shown in
Figure~\ref{fig:P10}. }

 Therefore, we may assume that there are exactly two distinct pairs of
sectors with jumps of length $1$ between them, say $x_1-y_1$ and
$y_2-x_2$ in order along $C$, where $u_g \in (x_1, y_1)$ and $u_h \in
(y_2, x_2)$.
 Suppose some sector has no jump of length $1$ out of it.  Without loss
of generality we may assume this sector is $U_0 \subseteq (x_2, x_1)$. 
There is some sector jump $x-y$ out of $U_0$.  Then $y \in [y_1, y_2]$,
otherwise Claim \ref{clm:jump1} would give a jump of length $1$ between
a third pair of sectors.  Therefore there is a $K_{2,5}$ minor
 $\apm(D \cup \set{u_0, u_1}, [y_1,y_2] ;
	x, x_1, u_g, u_h, x_2)$%
 \whenfig{ as shown in Figure~\ref{fig:P6}}.

 Therefore, every sector has a jump of length $1$ out of it, which means
that $k=4$, and we may assume that there are jumps $U_3-U_0$ and
$U_1-U_2$, but no jumps $U_0-U_1$ or $U_2-U_3$.
 Let $z_3-z_0$ be the
sector jump $U_3-U_0$ such that $z_3$ is closest to $u_3$ and $z_0$ is
closest to $u_1$. Similarly, let $z_1-z_2$ be the sector jump $U_1-U_2$
such that $z_1$ is closest to $u_1$ and $z_2$ is closest to $u_3$.
 Each $U_j$ is divided into two parts by $z_j$: let $A_0=(u_0,z_0)$,
$B_0=(z_0,u_1)$, $B_1=(u_1,z_1)$, $A_1=(z_1,u_2)$, $A_2=(u_2,z_2)$,
$B_2=(z_2,u_3)$, $B_3=(u_3,z_3)$ and $A_3=(z_3,u_0)$.

 We may assume that $z_3-z_0$ and $z_1-z_2$ are embedded in the plane so
that $D$ is outside both cycles $Z_0=C[z_3,z_0] \cup z_3-z_0$ and
$Z_2 = C[z_1, z_2] \cup z_1-z_2$.
 Let $H_0$ be the subgraph of $G$ consisting of $Z_0$ and all
vertices and edges inside $Z_0$, and define $H_2$ similarly; these are
$2$-connected by Lemma \ref{lem:insidecycle}.

 For any $j$, define $N_j$ to be the set of vertices of $V(G)-V(C)-V(D)$
inside a cycle $C[u_j, u_{j+1}] \cup u_{j+1} D u_j$ (the exact path
through $D$ does not matter).  Loosely, these are the vertices inside $C$
associated with the sector $U_j$.
 We now claim that there is a $K_{2,2}$ minor along $[u_3, u_1]$ using
only vertices in $[u_3, u_1] \cup V(H_0) \cup N_3 \cup N_0$.

 %  If $N_3 \ne \emptyset$, then there are three internally disjoint paths
 % from $v \in N_3$ to $V(C)$, and we may assume that no internal vertices
 % of these paths belong to $C$, and hence these paths are disjoint from
 % $D$.  The ends of these paths other than $v$ must belong to $[u_3,u_0]$.
 %  Suppose these ends are $w_1, w_2, w_3$ in order along $C$.  Then
 % $\apm([u_3, w_1], [w_3, u_1] ; w_2, v)$ is the required $K_{2,2}$ minor.
 %  Thus, we may assume that $N_3 = \emptyset$, and symmetrically that $N_0
 % = \emptyset$.
 
 If $N_3 \ne \emptyset$, then there is a component $D'$ of $G-V(C)$ with
$V(D') \subseteq N_3$.  Now $D'$ has (at least) three neighbors in
$[u_3,u_0]$, say $w_1, w_2, w_3$ in order along $C$.
 So $\apm([u_3, w_1], [w_3, u_1] ; w_2, D')$ is the required $K_{2,2}$
minor.
 Thus, we may assume that $N_3 = \emptyset$, and symmetrically that $N_0
= \emptyset$.
 
\begin{figure}
\centering
\begin{minipage}[b]{.33\linewidth}
\centering \scalebox{.95}{\includegraphics{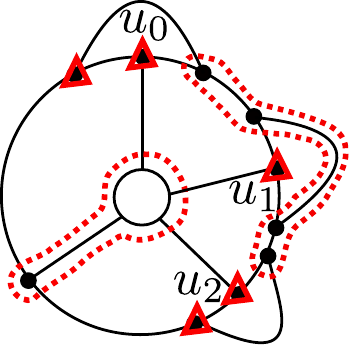}}
\captionof{figure}{\label{fig:P10}}
\end{minipage}%
\begin{minipage}[b]{.33\linewidth}
\centering \scalebox{.95}{\includegraphics{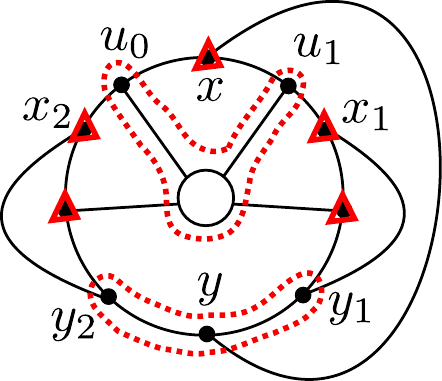}}
\captionof{figure}{\label{fig:P6}}
\end{minipage}%
\begin{minipage}[b]{.33\linewidth}
\centering \scalebox{.95}{\includegraphics{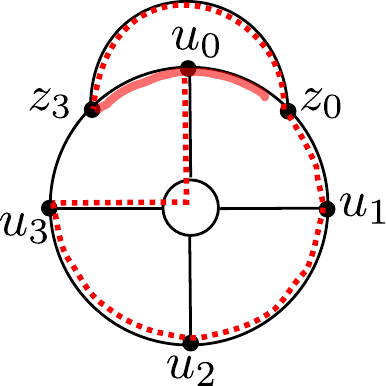}}
\captionof{figure}{\label{fig:P22}}
\end{minipage}%
\end{figure}

 Let $H_0' = H_0 \cup G[z_3, z_0]$.
 Then $V(H_0') = V(H_0)$, so $H_0'$ is also $2$-connected, but possibly
$E(H_0') \ne E(H_0)$ because $H_0'$ contains any edges inside $C$
joining two vertices of $[z_3,u_0]$ or two vertices of $[u_0,z_0]$.
 If $H_0'$ has a $K_{2,2}$ minor rooted at $z_3$ and $z_0$, such as a
minor $\apm(z_3, z_0; u_0, q)$ when $z_3-z_0$ has an internal vertex
$q$, then we can extend this minor using $[u_3,z_3]$ and $[z_0,u_1]$ to
get the required $K_{2,2}$ minor. 
 If there is an inside jump out of any of $B_3, A_3, A_0, B_0$, then
this jump together with $z_3-z_0$ forms the required $K_{2,2}$ minor.

 %  So we may assume that $H_0'$ has no $K_{2,2}$ minor rooted at $z_3$ and
 % $z_0$.  Thus, $z_3-z_0$ has no internal vertex and so $z_3 z_0$ is an
 % outer edge of $H_0'$.
 %  Also, by Lemma \ref{lem:rooted22}, $H_0'$ is $z_3z_0$-outerplanar.
 %  If a vertex of $A_3$ has an incident edge not in $H_0'$, then since
 % $N_3 = \emptyset$ that edge is an inside jump, creating the required
 % $K_{2,2}$ minor.  So we may assume that all edges incident with $A_3$,
 % and symmetrically with $A_0$, belong to $H_0'$.
 %  Thus, the only vertices that could possibly have degree $2$ in $H_0'$
 % are $z_3, u_0$ and $z_0$.

 So we may assume that $H_0'$ has no $K_{2,2}$ minor rooted at $z_3$ and
$z_0$.  Thus, $z_3-z_0$ has no internal vertex and so $z_3 z_0$ is an
outer edge of $H_0'$.
 Also, by Lemma \ref{lem:rooted22}, $H_0'$ is $z_3z_0$-outerplanar. 
 If there is an edge of $G$ leaving $H_0'$ at a vertex of $A_3$ or $A_0$
then, since $N_3=N_0=\emptyset$, that edge is an inside jump, creating
the required $K_{2,2}$ minor.
 Hence, any edges of $G$ leaving $H_0'$ leave at $z_3$, $u_0$ or $z_0$.
 Since $G$ is $3$-connected these are the only vertices that can have
degree $2$ in $H_0'$.

 Suppose that $B_3 = \emptyset$.  By Lemma \ref{lem:P0} there is a
Hamilton path $P = z_0 z_3 \ldots t$ in $H_0'$ where $t$ has degree $2$
in $H_0'$; then we must have $t = u_0$.  Thus, $P \cup C[z_0,u_3] \cup
u_3Du_0$ is a longer cycle, a contradiction.
 This cycle is shown in Figure \ref{fig:P22}, where we use the convention
that paths found using Lemma \ref{lem:P0} or Corollary
 \ref{cor:P1} are shown by heavily shading the part of the graph
covered by the paths; the rest of the cycle is shown using dotted
curves.
 Thus, $B_3$ is nonempty, and by a symmetric argument $B_0$ is also
nonempty.

 Suppose $r_0-t$ is an outside jump out of $B_0$.
 This jump cannot contain an internal
vertex of $z_3-z_0$, and $t \notin (u_3, z_3]$, by choice of $z_3-z_0$.
 The jump cannot contain an internal vertex of $z_1-z_2$, and 
$t \notin (u_1,z_1]$, because there are no $U_0-U_1$ jumps.
 Thus, $t \in [z_2, u_3]$.
 Similarly, an outside jump $r_3-t'$ out of $B_3$ must have $t' \in [u_1,z_1]$. 
 Hence we cannot have outside jumps out of both $B_0$ and $B_3$ because
the jumps $r_0-[z_2,u_3]$ and $r_3-[u_1,z_1]$ would intersect by
planarity, giving a jump $r_3-r_0$ that contradicts the choice of
$z_3-z_0$.  Therefore, there is an inside jump out of one of $B_0$ or
$B_3$, giving the required $K_{2,2}$ minor along $[u_3,u_1]$.

 By a symmetric argument there is also a $K_{2,2}$ minor along
$[u_1,u_3]$ using only vertices in $[u_1, u_3] \cup V(H_2) \cup N_1 \cup
N_2$.  The two minors intersect only at $u_1$ and $u_3$, so together
they give a $K_{2,5}$ minor $M(u_3, u_1)$.
 This concludes the proof of Claim \ref{clm:1}.

 \medskip
Henceforth we assume $k=3$. The next claim simplifies the structure of
the graph we are looking at and makes further analysis easier. 

 \begin{claim} Without loss of generality, we may assume
 that $D$ consists of a single degree $3$ vertex $d$ and that $V(G) = V(C)
\cup \{d\}$.  Thus, every jump is a single edge.
 We may also assume that there are no edges $xy \in E(G)-E(C)$ where $G$
has three internally disjoint $xy$-paths of length $2$ or more; in
particular $u_i u_j \notin E(G)$ for all $i, j  \in \set{0, 1, 2,
\ldots, k-1}$.
 \label{clm:3}
 \end{claim}

 Since $k=3$ and $G$ is $3$-connected, every component of $G-V(C)$ has
exactly three neighbors on $C$.
 Applying Lemma \ref{lem:redcomp} to each of these components in turn,
including $D$, we find that $G$ is $C$-reducible to $G_1$ for which
every component of $G_1-V(C)$ is a single degree $3$ vertex of $G_1$. 
Let $d$ be the degree $3$ vertex corresponding to $D$.
 Applying Lemma \ref{lem:reddeg3} to each vertex of
$V(G_1)-V(C)-\set{d}$ in turn, we find that $G_1$ is $C$-reducible to
$G_2$ for which $V(G_2) = V(C) \cup \{d\}$.
 Starting from $G_2$ and applying Lemma \ref{lem:rededge} repeatedly to
any edge $xy$ not on $C$ where there are three internally disjoint
$xy$-paths of length $2$ or more, we find that $G_2$ is $C$-reducible to
$G_3$ in which there are no such edges $xy$.  Since $u_iu_j \notin E(C)$
for all $i$ and $j$, $G_3$ has no edges $u_i u_j$ by the second part of
Lemma \ref{lem:rededge} .

 Since $C$-reducibility is transitive, $G$ is $C$-reducible to $G_3$. 
$G_3$ is $3$-connected and has all the properties stated in the claim. 
Since $G_3$ is a minor of $G$, $G_3$ is planar, and showing that $G_3$
has a $K_{2,5}$ minor also shows that $G$ has a $K_{2,5}$ minor.  By (c)
of the definition of $C$-reducibility, showing that $G_3$ has a cycle
longer than $C$ also shows that $G$ has a cycle longer than $C$. 
Therefore, we may replace $G$ by $G_3$ in our arguments.  This concludes
the proof of Claim \ref{clm:3}.

 \medskip

 We are now in the general situation
where there are three sectors labeled $U_0$, $U_1$, and $U_2$. Let
$t_0-t_1$ be the outermost $U_2-U_0$ jump (if any $U_2-U_0$ jump
exists), meaning that $t_0 \in U_2$ is closest to $u_2$ and $t_1 \in
U_0$ is closest to $u_1$. Similarly let $t_2-t_3$ be the outermost
$U_0-U_1$ jump, and $t_4-t_5$ the outermost $U_1-U_2$ jump, when such
jumps exist.
 Because every sector must have a jump out
of it and by Claim \ref{clm:jump1}, there are at least two sector jumps
of length $1$; without loss of generality, assume there are
jumps $t_0-t_1$ and $t_2-t_3$. Define
$X_0=(t_0,u_0)$, $X_1=(u_0,t_1)$, $X_2=(t_2,u_1)$, $X_3=(u_1,t_3)$,
$X_4=(t_4,u_2)$, and $X_5=(u_2,t_5)$, whenever the necessary $t_i$
vertices exist.
 \whenfig{An example of the overall situation is shown in
Figure~\ref{fig:genform}.}

\begin{figure}
\centering
\begin{minipage}[b]{.33\linewidth}
\centering \scalebox{.95}{\includegraphics{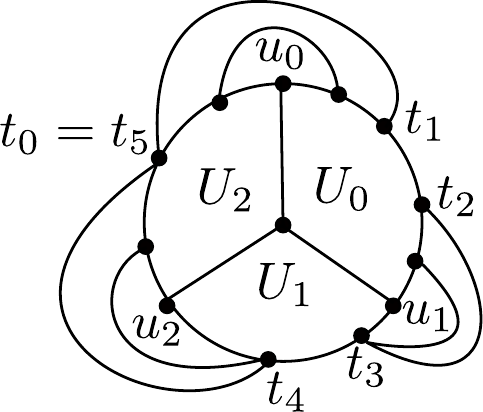}}
\captionof{figure}{\label{fig:genform}}
\end{minipage}%
\begin{minipage}[b]{.33\linewidth}
\centering \scalebox{.95}{\includegraphics{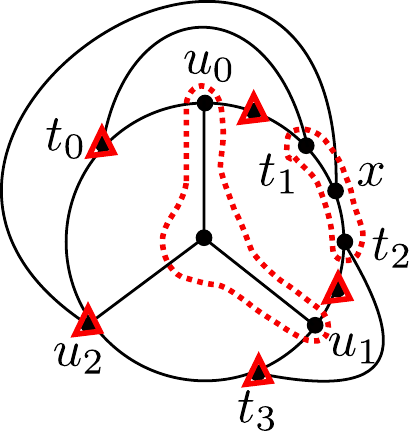}}
\captionof{figure}{\label{fig:C1}}
\end{minipage}%
\begin{minipage}[b]{.33\linewidth}
\centering \scalebox{.95}{\includegraphics{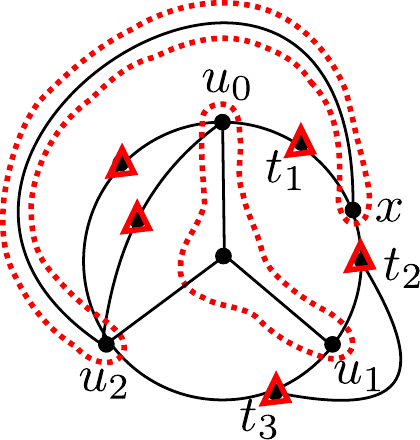}}
\captionof{figure}{\label{fig:C2}}
\end{minipage}%
\end{figure}

\begin{claim} There are no sector jumps $x-u_2$ where $x \in (t_1,t_2)$.
\label{clm:u2}
\end{claim}

 Let $x-u_2$ be a sector jump with $x \in (t_1,t_2)$.
 If there exist $q_1 \in X_1$ and $q_2 \in X_2$ then there is a
$K_{2,5}$ minor
 $\apm(\set{d,u_0,u_1}, [t_1,t_2];
	q_1, q_2, t_3, u_2, t_0)$%
 \whenfig{ as shown in Figure~\ref{fig:C1}}.
 So at least one of $X_1$ and $X_2$ is empty; without loss of
generality, assume $X_1=\emptyset$.
 Since $X_1=\emptyset$ and by choice of $t_0-t_1$, all
jumps out of $U_2$ must go to $t_1$.
 If there is a $K_{2,2}$ minor $\apm(u_2, u_0; s_1, s_2)$ along
$[u_2,u_0]$, then there is a $K_{2,5}$ minor
 $\apm(\set{d, u_0, u_1}, x-u_2;
	t_1, t_2, t_3, s_1, s_2)$%
 \whenfig{ as shown in Figure~\ref{fig:C2}}.
 So we may assume there is no such minor, and apply
Corollary~\ref{cor:P1} to $G[u_2,u_0]$ to find a path $P=u_0\ldots t$
such that $V(P)=(u_2,u_0]$ and $t$ is a degree $2$ vertex in
$G[u_2,u_0]$; then we must have $tt_1 \in E(G)$.
 Thus, $P \cup tt_1 \cup C[t_1, u_2] \cup u_2du_0$ is a longer cycle%
 \whenfig{, as shown in Figure~\ref{fig:C3}}.
 This completes the proof of Claim~\ref{clm:u2}.

\begin{figure}
\centering
\begin{minipage}[b]{.33\linewidth}
\centering \scalebox{.95}{\includegraphics{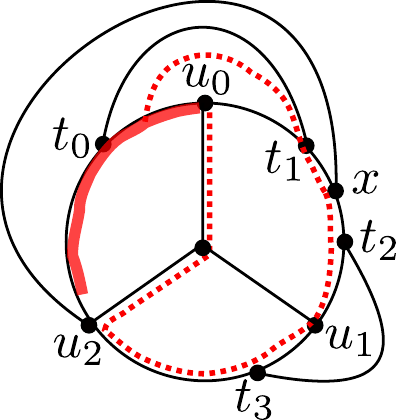}}
\captionof{figure}{\label{fig:C3}}
\end{minipage}%
\begin{minipage}[b]{.33\linewidth}
\centering \scalebox{.85}{\includegraphics{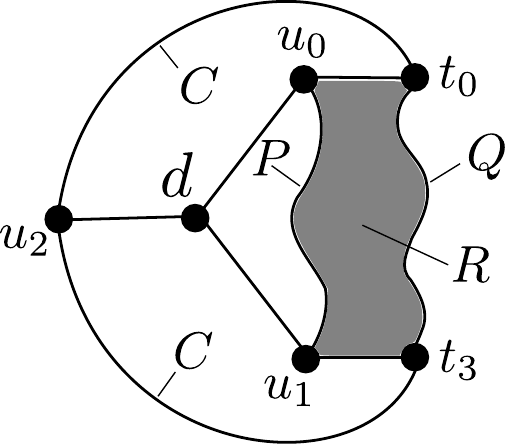}}
\captionof{figure}{\label{fig:Case1base}}
\end{minipage}%
\end{figure}

\begin{claim} Either $t_0 \neq u_0^-$ or $t_3 \neq u_1^+$ ($X_0$ and
$X_3$ cannot both be empty).
\label{clm:4}
\end{claim}

 % Figure Case1base is not optional, always include.

 Assume that $t_0=u_0^-$ and $t_3=u_1^+$.
 See Figure~\ref{fig:Case1base}.
 Let $R=G[t_0,t_3]$; we may assume that $d$ is outside $R$.
 There are three internally disjoint $t_0t_3$-paths of length $2$ or
more, namely $t_0-t_1 \cup C[t_1,t_2] \cup t_2-t_3$, $t_0 u_0 d u_1 t_3$
and $C[t_3, t_0]$, so by Claim \ref{clm:3}, $t_0 t_3 \notin E(G)$. Also
by Claim \ref{clm:3}, $u_0 u_1 \notin E(G)$.

 Let $P$ be the walk from $u_0$ to $u_1$ counterclockwise along the
outer face of $R$ and $Q$ be the walk from $t_0$ to $t_3$ clockwise
along the outer face of $R$.  The outer face of $R$ is bounded by $P
\cup Q \cup \set{u_0t_0, u_1t_3}$.
 If $P=(u_0=p_0) p_1 p_2 \ldots p_{r-1} (p_r=u_1)$ then each $p_i$, $1 \le i \le
r$, is closer to $t_3$ along $C[t_0,t_3]$ than $p_{i-1}$, so $P$ has no
repeated vertices and is a path; similarly, $Q$ is a path.
 Additionally, $|V(P)| \geq 3$ because $u_0u_1 \notin E(G)$ and $|V(Q)|
\geq 3$ because $t_0t_1, t_2t_3 \in E(Q)$ (possibly $t_1=t_2$).

 % *** Is there any problem here caused by vertices inside $P$?
 % No, at this point all vertices are d or on C.

 The paths $P$ and $Q$ may intersect but only in limited ways.
 Any intersection vertex must belong to $[t_1, t_2]$.
 If $P$ and $Q$ intersect at two non-consecutive vertices on $C$, then
using Claim~\ref{clm:u2} these two vertices would form a $2$-cut in $G$.
Hence there are three possibilities for $P$ and $Q$: $V(P) \cap
V(Q)=\{x,x^+\}$, $V(P) \cap V(Q)=\{x\}$, or $V(P) \cap V(Q) =
\emptyset$.

 \smallskip
 (1) First assume $V(P) \cap V(Q)=\{x,x^+\} \subseteq [t_1, t_2]$.
 We will show that there is a longer cycle.
 Let $R_1=G[t_0,x]$ and $R_2=G[x^+,t_3]$.
 Then $t_0t_1 \in E(R_1)$ and $t_2t_3 \in E(R_2)$.
 Let $P_1 = P \cap R_1$ and $Q_1 = Q \cap R_1$; then $u_0 \in V(P_1)$,
$t_0t_1 \in E(Q_1)$, and $V(P_1) \cap V(Q_1) = \set{x}$.
 First we construct a new $u_0x$-path $P'_1$ and a new $t_0x$-path
$Q'_1$ such that $V(P'_1 \cup Q'_1) = V(R_1)$ and $V(P'_1) \cap
V(Q'_1) = \set{x}$.
 If $Q_1$ is just the edge $t_0t_1$ (so $t_1=x$) we may take $P'_1 =
C[u_0, x]$ and $Q'_1 = Q_1$.  So we may assume that $|V(Q_1)| \ge 3$.
 
 Let $P'_1$ be a $u_0x$-path in $R_1$ and $Q'_1$ a $t_0x$-path in $R_1$
so that $V(P_1) \subseteq V(P'_1)$, $V(Q_1) \subseteq V(Q'_1)$ and
$V(P'_1) \cap V(Q'_1) = \set{x}$.  Such paths exist since we can take
$P'_1=P_1$ and $Q'_1=Q_1$.  Additionally assume  $|V(P'_1) \cup
V(Q'_1)|$ is maximum.
 Suppose $V(P'_1 \cup Q'_1) \neq V(R_1)$ and let $K$ be a component of
$R_1 - V(P'_1 \cup Q'_1)$.
 Because $G$ is $3$-connected, $K$ must have at least three neighbors in
$G$.
 Since $V(P_1 \cup Q_1) \subseteq V(P'_1 \cup Q'_1)$, $K$ contains
no external vertices of $R_1$.
 Therefore, by planarity all neighbors of $K$ are in $R_1$ and hence in
$V(P'_1 \cup Q'_1)$.
 Thus, $K$ has at least two neighbors in one of $P'_1$ or $Q'_1$.

 Suppose first that $K$ is adjacent to $w_1, w_2 \in V(Q'_1)$.
 If $w_1w_2 \in E(Q'_1)$, then we can lengthen $Q'_1$ (still with
$V(Q_1) \subseteq V(Q'_1)$): replace the edge $w_1w_2$ with a path from
$w_1$ to $w_2$ through $K$.
 Hence we may assume that $Q'_1 = t_0 \ldots w_1 \ldots w_3 \ldots w_2
\ldots x$ with $w_3 \ne w_1, w_2$, and we have a
 $K_{2,5}$ minor
 $\apm([u_2, t_0] \cup Q'_1[t_0,w_1], Q'_1[w_2,x] \cup (P \cap R_2);
	d, u_0, K, w_3, t_3)$%
 \whenfig{, a special case of which is shown in Figure~\ref{fig:F2}}.
 Suppose now that $K$ is adjacent to $w_1, w_2 \in V(P'_1)$.
 If $w_1w_2 \in E(P'_1)$ then we can lengthen $P'_1$,
 so $P'_1 = u_0 \ldots w_1 \ldots w_3 \ldots w_2 \ldots x$ with $w_3 \ne
w_1, w_2$, and we have a
 $K_{2,5}$ minor
 $\apm([u_2, u_0] \cup P'_1[u_0,w_1], P'_1[w_2,x] \cup (P \cap R_2);
	d, w_3, K, y, t_3)$
 where $y$ is an internal vertex of $Q'_1$, which exists because
$|V(Q'_1)| \ge |V(Q_1)| \ge 3$.
 Thus no such component $K$ exists, $V(P'_1 \cup Q'_1) = V(R_1)$, and
$P'_1$ and $Q'_1$ are the desired paths in $R_1$.
 
 By symmetric arguments, $R_2$ has a $u_1x^+$-path $P'_2$
and a $t_3x^+$-path $Q'_2$ such that $V(P'_2 \cup
Q'_2)=V(R_2)$ and $V(P'_2) \cap V(Q'_2)=\set{x^+}$.
 Hence there is a longer cycle $C[t_3,t_0] \cup Q'_1 \cup P'_1 \cup u_0du_1
\cup P'_2 \cup Q'_2$%
 \whenfig{ as shown in Figure~\ref{fig:F1}}.

\begin{figure}
\centering
\begin{minipage}[b]{.33\linewidth}
\centering \scalebox{.85}{\includegraphics{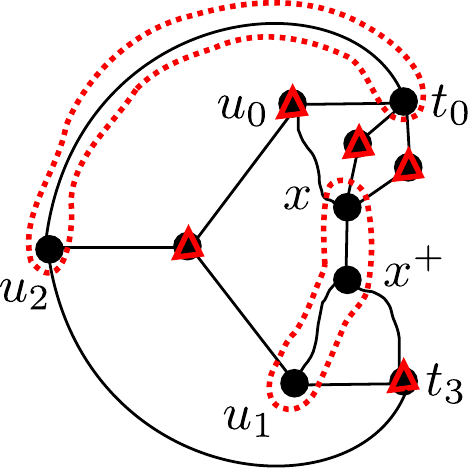}}
\captionof{figure}{\label{fig:F2}}
\end{minipage}%
\begin{minipage}[b]{.33\linewidth}
\centering \scalebox{.85}{\includegraphics{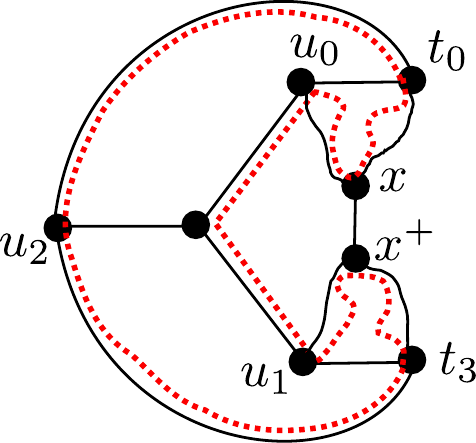}}
\captionof{figure}{\label{fig:F1}}
\end{minipage}%
\begin{minipage}[b]{.33\linewidth}
\centering \scalebox{.85}{\includegraphics{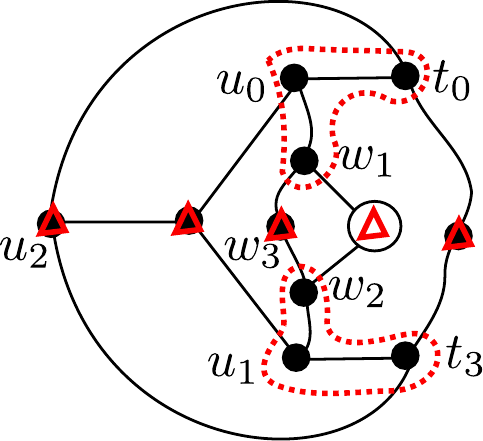}}
\captionof{figure}{\label{fig:F3}}
\end{minipage}%
\end{figure}

 \smallskip
 (2) Assume $V(P) \cap V(Q)=\{x\} \subseteq [t_1, t_2]$.
 The argument here will be very similar to, but not exactly the same as,
that in (1).
 Again let $R_1=G[t_0,x]$ and $R_2=G[x^+,t_3]$. Then $t_0t_1 \in E(R_1)$
and $t_2 \in V(R_2)\cup\{x\}$.
 Using the argument from (1), in $R_1$ we find a $u_0x$-path $P'_1$ and
a $t_0x$-path $Q'_1$ with $V(P'_1) \cup V(Q'_1) = V(R_1)$ and $V(P'_1)
\cap V(Q'_1) = \set{x}$.

 We also want to find in $R_2$ a $u_1x^+$-path $P'_2$ and a
$t_3x^+$-path $Q'_2$ such that $V(P'_2 \cup Q'_2)=V(R_2)$ and $V(P'_2)
\cap V(Q'_2)=\set{x^+}$, but this requires some changes from (1).
 Let $P_2$ be the segment of the outer boundary of $R_2$ clockwise from
$u_1$ to $x^+$, and let $Q_2$ be the segment counterclockwise from $t_3$
to $x^+$.  Then $P_2$ and $Q_2$ are paths by the same argument as for
$P$ and $Q$.
 If there is an edge $t_3x^+$ (including when $x^+=u_1$) then we can
take $P'_2 = C[x^+,u_1]$ and $Q_2'=Q_2=t_3x^+$, so we may assume there
is no such edge and hence $|V(Q_2)| \ge 3$.

 Assume there is $v \in (V(P_2) \cap V(Q_2)) - \set{x^+}$.
 Using Claim~\ref{clm:u2}, every edge leaving $(x,v)$ (which contains
$x^+$) goes to $x$ or $v$, or is the edge $t_2 u_2$.
 But since $t_2$ is adjacent to $t_3$, $t_2 \notin (x,v)$
 so $\{x, v\}$ is a $2$-cut in $G$, a contradiction.
 Thus, $V(P_2) \cap V(Q_2) = \set{x^+}$.  Now we have a $u_1x^+$-path
$P_2$ and a $t_3x^+$-path $Q_2$ so that (a) all external vertices of
$R_2$ belong to $V(P_2 \cup Q_2)$, (b) $V(P_2) \cap V(Q_2) = \set{x^+}$,
and (c) $|V(Q_2)| \ge 3$.
 This allows us to apply the argument for $|V(Q_2)| \ge 3$ from (1) to
find the required $P'_2$ and $Q'_2$ in $R_2$.

 As in (1), we use $P'_1, Q'_1, P'_2, Q'_2$ to find a
longer cycle.

 \smallskip
 (3) Finally suppose $V(P) \cap V(Q) = \emptyset$. Let $P'$ be a
$u_0u_1$-path in $R$ and $Q'$ a $t_0t_3$-path in $R$ so that $V(P)
\subseteq V(P')$, $V(Q) \subseteq V(Q')$ and $V(P') \cap V(Q') =
\emptyset$.  
 Such paths exist because we can take $P'=P$ and $Q'=Q$. Assume
additionally that $|V(P') \cup V(Q')|$ is maximum.
 Suppose $V(P' \cup Q') \neq V(R)$ and let $K$ be a component of $R-V(P'
\cup Q')$.
 Because $G$ is $3$-connected, $K$ must have at least three neighbors in
$G$.
 Since $V(P \cup Q) \subseteq V(P' \cup Q')$, $K$ contains
no external vertices of $R$.
 Therefore, by planarity all neighbors of $K$ are in $R$ and hence in
$V(P' \cup Q')$.
 Thus, $K$ has at least two neighbors in one of $P'$ or $Q'$.

 First suppose $K$ is adjacent to $w_1,w_2 \in V(P')$.
 If $w_1w_2 \in E(P')$, then we can lengthen $P'$ (still with
$V(P) \subseteq V(P')$): replace the edge $w_1w_2$ with a path from
$w_1$ to $w_2$ through $K$.
 Hence we may assume that $P' = u_0 \ldots w_1 \ldots w_3 \ldots w_2
\ldots u_1$ with $w_3 \ne w_1, w_2$, and we have a
 $K_{2,5}$ minor
 $\apm(u_0 t_0 \cup P'[u_0,w_1], P'[w_2,u_1] \cup u_1t_3;
	u_2, d, w_3, K, y)$%
 \whenfig{ as shown in Figure~\ref{fig:F3}},
 where $y$ is an internal vertex of $Q'$, which exists because $|V(Q')|
\ge |V(Q)| \ge 3$.
 We can reason similarly if $K$ is adjacent to $w_1, w_2 \in V(Q')$.  
 Thus no such component $K$ exists and $V(P' \cup Q') = V(R)$.

 Suppose there is a $K_{2,2}$ minor
 $\apm(t_0u_0, u_1t_3; s_1, s_2)$
 in $G[u_1, u_0]$.
 Then there is a $K_{2,5}$ minor
 $\apm(t_0u_0, u_1t_3;
	s_1, \allowbreak s_2, \allowbreak d, p, q)$%
 \whenfig{ as shown in Figure~\ref{fig:F5}},
 where $p$, $q$ are arbitrary internal vertices of $P$, $Q$ respectively.
 So we may assume there is no such minor.  Therefore, there is no
$K_{2,2}$ minor along $[t_3,t_0]$ or any of its subintervals.

 Suppose that $(u_2,t_0) = \emptyset$ or all jumps out of $(u_2,t_0)$ go
to $u_0$.
 Apply Corollary~\ref{cor:P1} to $G[u_2,t_0]$ to find a path
$J=t_0\ldots t$ such that $V(J)=(u_2,t_0]$ and either $t=t_0$ if $(u_2,
t_0) = \emptyset$, or $t$ is a vertex of degree $2$ in $G[u_2,t_0]$,
from which there must be a jump to $u_0$.
 In either case, $t_0u_0 \in E(G)$ and there is a longer cycle
	$P' \cup u_1du_2 \cup C[t_3,u_2] \cup Q' \cup J \cup tu_0$%
 \whenfig{; the case when $(u_2,t_0) \neq \emptyset$ is
shown in Figure~\ref{fig:F4}}.

 So we may assume that not all jumps
out of $(u_2,t_0)$ go to $u_0$ and so there is a jump $x_1-x_2$ with
$x_1 \in (u_2, t_0)$ and $x_2 \in [t_3,u_2)$.
 By a symmetric argument there is also a jump $x_2'-x_1'$ with $x_2' \in
(t_3,u_2)$ and $x_1' \in (u_2, t_0]$.  These jumps cannot cross because
they are just edges, so we cannot have both $x_2=t_3$ and $x_1'=t_0$. 
Without loss of generality, $x_2 \ne t_3$, so $x_1-x_2$ is a jump from
$(u_2,t_0)$ to $(t_3,u_2)$.  Out of all such jumps we may assume that
$x_1-x_2$ has $x_1$ closest to $t_0$ and $x_2$ closest to $t_3$.

 If there is a jump $y_1-y_2$ from $(u_2,x_1)$ to $(x_1,u_0]$, then
$x_1-x_2$ and $y_1-y_2$ give a $K_{2,2}$ minor in $G[u_1,u_0]$ that we excluded above, namely
 $\apm([u_1, x_2], [y_2, u_0]; 	x_1, y_1)$ if $y_2 \ne u_0$, or
 $\apm([u_1, x_2], t_0u_0; x_1, y_1)$ if $y_2 = u_0$.
 A symmetric minor exists if there is a jump from $(x_2,u_2)$ to
$[u_1,x_2)$.
 Hence edges of $G$ leaving $G[x_2,x_1]$ leave at $x_1$, $x_2$ or $u_2$.
 Since $G[x_2,x_1]$ is bounded by the cycle $C[x_2,x_1] \cup
x_1x_2$, $G[x_2,x_1]$ is $2$-connected by Lemma \ref{lem:insidecycle}.
 Apply Lemma~\ref{lem:P0} to $G[x_2,x_1]$ to find a path
$J_1=x_2x_1\ldots t$ where $V(J_1)=[x_2,x_1]$ and $t$ is a degree $2$
vertex in $G[x_2,x_1]$ and hence must be $u_2$.
 Apply Corollary~\ref{cor:P1} to $G[x_1,t_0]$ to find a path
$J_2=t_0\ldots s$ where $V(J_2)=(x_1,t_0]$ and either $s=t_0$ or $s$ is
a degree $2$ vertex in $G[x_1,t_0]$.  In either case $su_0 \in E(G)$ and
there is a longer cycle
 $P' \cup u_1du_2 \cup J_1 \cup C[t_3,x_2] \cup Q' \cup J_2 \cup su_0$%
 \whenfig{, as shown in Figure~\ref{fig:F7}}.

\begin{figure}
\centering
\begin{minipage}[b]{.33\linewidth}
\centering \scalebox{.85}{\includegraphics{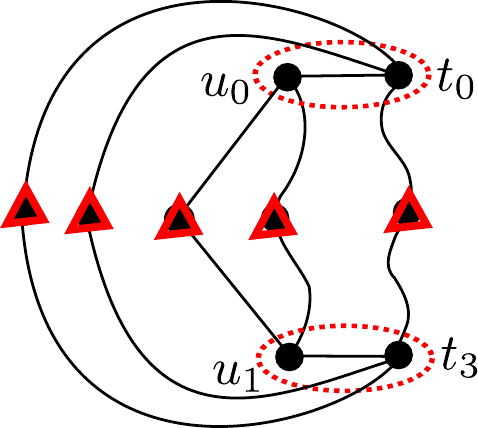}}
\captionof{figure}{\label{fig:F5}}
\end{minipage}%
\begin{minipage}[b]{.33\linewidth}
\centering\scalebox{.85}{\includegraphics{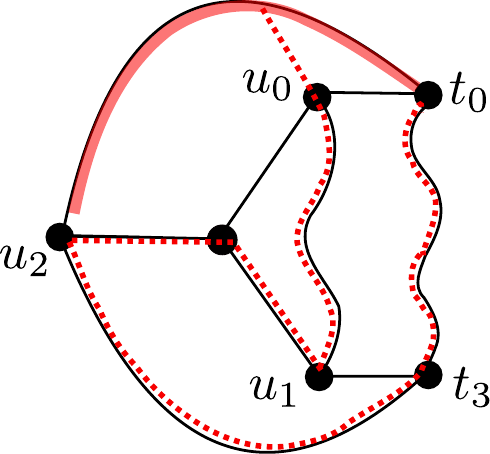}}
\captionof{figure}{\label{fig:F4}}
\end{minipage}%
\begin{minipage}[b]{.33\linewidth}
\centering \scalebox{.85}{\includegraphics{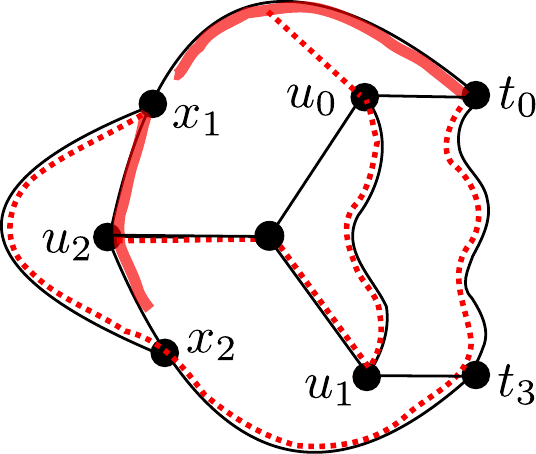}}
\captionof{figure}{\label{fig:F7}}
\end{minipage}%
\end{figure}

This completes the proof of Claim~\ref{clm:4}.

 \medskip

\begin{claim} Either $t_1= u_0^+$ or $t_2 = u_1^-$ (at least one of
$X_1$ and $X_2$ is empty).
\label{clm:5}
\end{claim}

 Assume $t_1\neq u_0^+$ and $t_2 \neq u_1^-$. By Claim~\ref{clm:4},
either $t_0 \neq u_0^-$ or $t_3 \neq u_1^+$. Without loss of generality,
suppose $t_0 \neq u_0^-$. Then there is a $K_{2,5}$ minor
 $\apm(u_0du_1, t_0t_1 \cup [t_1, t_2];
	u_0^-, u_0^+, u_1^-, t_3, u_2)$%
 \whenfig{ as shown in Figure~\ref{fig:Picture2}}.

\begin{claim} At most two pairs of sectors have jumps between them.
\label{clm:6}
\end{claim}

 Assume that there are three sector jumps $t_0-t_1$, $t_2-t_3$, and
$t_4-t_5$ where possibly $t_0=t_5$, $t_1=t_2$, or $t_3=t_4$.
 By Claim~\ref{clm:4}, $X_0$ and $X_3$ cannot both be empty and
symmetrically, $X_1$ and $X_4$ cannot both be empty and $X_2$ and $X_5$
cannot both be empty. Hence $X_i \neq \emptyset$ for at least three $i$.
 By Claim~\ref{clm:5}, at least one of $X_1$ and $X_2$ is empty and
symmetrically, at least one of $X_3$ and $X_4$ is empty and at least one
of $X_5$ and $X_0$ is empty. Hence $X_i \neq \emptyset$ for exactly
three $i$. Furthermore, the nonempty $X_i$ must be rotationally
symmetric about $C$. Without loss of generality, suppose $X_0$, $X_2$,
and $X_4$ are nonempty and $X_1$, $X_3$, and $X_5$ are empty.

\begin{figure}
\centering
\begin{minipage}[b]{.33\linewidth}
\centering \scalebox{.95}{\includegraphics{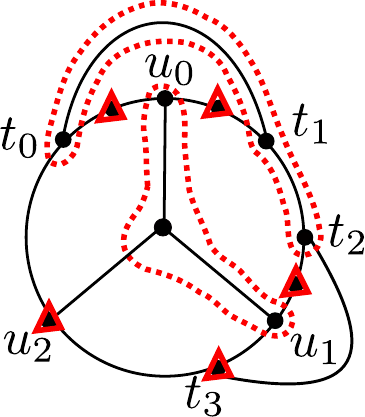}}
\captionof{figure}{\label{fig:Picture2}}
\end{minipage}%
\begin{minipage}[b]{.33\linewidth}
\centering \scalebox{.95}{\includegraphics{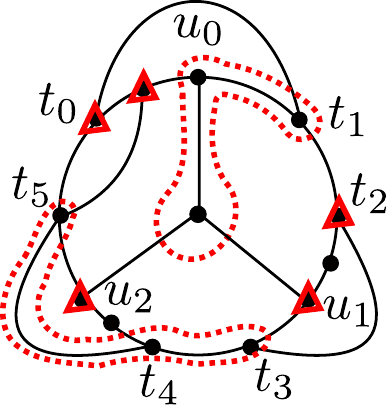}}
\captionof{figure}{\label{fig:R8}}
\end{minipage}%
\end{figure}

If $t_1=t_2$, then there is a standard longer cycle $L(u_0^+-u_1^+)$.
 Thus $t_1 \ne t_2$, and symmetrically $t_3 \ne t_4$ and $t_5 \ne t_0$.

 Consider a jump $r_0-r_0'$ out of $X_0$. There are three options for
$r_0'$: $r_0' \in [t_5,t_0)$, $r_0'=t_1$, or $r_0'=u_2$.
 If $r_0' \in [t_5,t_0)$ then, since $t_1 \neq t_2$, there is a $K_{2,5}$
minor
 $\apm(du_0t_1, [t_3,t_4] \cup t_4t_5 \cup [t_5,r_0'];
	t_0, r_0, t_2, u_1, u_2)$%
 \whenfig{; the case $r_0'=t_5$ is shown in Figure~\ref{fig:R8}}.
 Thus, $r_0' \in \set{u_2, t_1}$, and symmetrically $r_2' \in \set{u_0,
t_3}$
 for a jump $r_2-r_2'$ out of $X_2$, and $r_4' \in \set{u_1, t_5}$ for a
jump $r_4-r_4'$ out of $X_4$.

\begin{figure}
\centering
\begin{minipage}[b]{.33\linewidth}
\centering \scalebox{.95}{\includegraphics{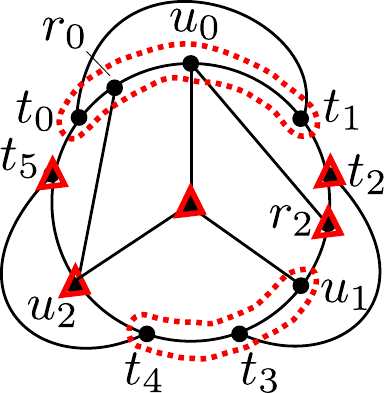}}
\captionof{figure}{\label{fig:R4}}
\end{minipage}%
\begin{minipage}[b]{.33\linewidth}
\centering \scalebox{.95}{\includegraphics{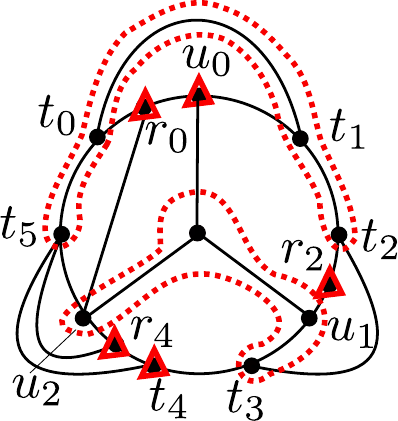}}
\captionof{figure}{\label{fig:R12}}
\end{minipage}%
\begin{minipage}[b]{.33\linewidth}
\centering \scalebox{.95}{\includegraphics{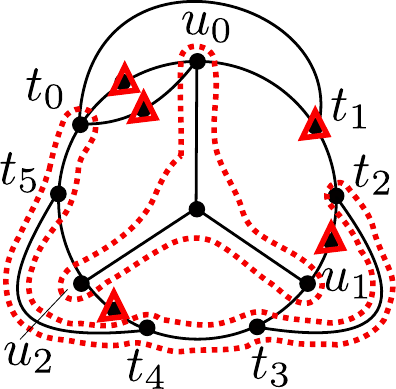}}
\captionof{figure}{\label{fig:R15}}
\end{minipage}%
\end{figure}

 If at least two of $r_0'$, $r_2'$, and $r_4'$ belong to
$U=\set{u_0,u_1,u_2}$ then without loss of generality we may assume that
$r_0'=u_2$ and $r_2'=u_0$.  We have a $K_{2,5}$ minor
 $M([t_0,t_1], [u_1,t_4])$%
 \whenfig{ as shown in Figure~\ref{fig:R4}}.
 If only one of $r_0'$, $r_2'$, and $r_4'$ belongs to $U$, then without
loss of generality $r_0'=u_2$ and there is a
$K_{2,5}$ minor
 $\apm(u_2du_1t_3, [t_5, t_0] \cup t_0t_1 \cup [t_1, t_2];
	r_0, u_0, r_2, t_4, r_4)$%
 \whenfig{ as shown in Figure~\ref{fig:R12}}.
 Hence we may assume that all jumps out of $X_0$ go to $t_1$, out of
$X_2$ go to $t_3$, and out of $X_4$ go to $t_5$.

 If there is a $K_{2,2}$ minor $\apm(t_0, u_0; s_1, s_2)$ along
$[t_0,u_0]$, then there is a $K_{2,5}$ minor
 $\apm(\set{d, u_0, u_1, u_2},
	t_2t_3 \cup [t_3,t_4] \cup t_4t_5 \cup [t_5,t_0];
	s_1, s_2, t_1, r_2, r_4)$%
 \whenfig{ as shown in Figure~\ref{fig:R15}}.
 Hence there is no $K_{2,2}$ minor along $[t_0,u_0]$, or symmetrically,
along $[t_2, u_1]$ or $[t_4, u_2]$.
 Because all jumps out of $X_4$ go to $t_5$
 we can apply Corollary~\ref{cor:P1} to $G[t_4,u_2]$ and find a path
$P=t_4\ldots t$ where $V(P) = [t_4,u_2)$ and $t$ has degree $2$ in
$G[t_4,u_2]$, so $tt_5 \in E(G)$.
 If $(t_5,t_0) = \emptyset$, then there is a longer cycle
 $C[t_0,u_0] \cup t_0t_1 \cup C[t_1,t_4] \cup P \cup tt_5u_2du_0$%
 \whenfig{ as shown in Figure~\ref{fig:R2}.}
 Hence $(t_5,t_0) \neq \emptyset$.
 Let $y-y'$ be a jump out of
$(t_5,t_0)$.
 Since all jumps out of $X_0$ go to $t_1$, $y' \notin X_0$, so $y' =
u_0$ or $u_2$.  Then there is a $K_{2,5}$ minor
 $\apm(yy' \cup u_0du_2, [t_1,t_2] \cup t_2t_3 \cup [t_3,t_4];
	t_0, r_0, u_1, r_4, t_5)$%
 \whenfig{; the case $y'=u_0$ is shown in Figure~\ref{fig:R14}}.
 This completes the proof of Claim~\ref{clm:6}.

 \medskip

\begin{figure}
\centering
\begin{minipage}[b]{.45\linewidth}
\centering \scalebox{.95}{\includegraphics{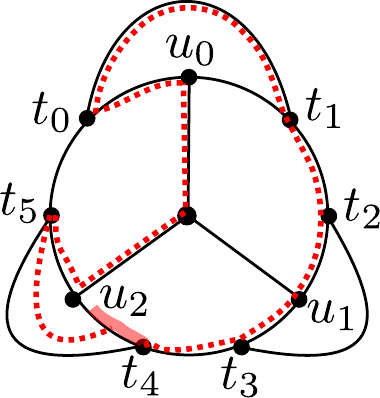}}
\captionof{figure}{\label{fig:R2}}
\end{minipage}%
\begin{minipage}[b]{.45\linewidth}
\centering \scalebox{.95}{\includegraphics{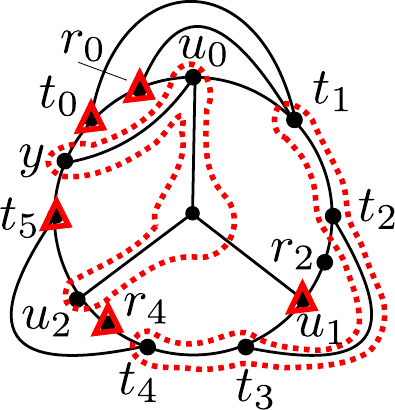}}
\captionof{figure}{\label{fig:R14}}
\end{minipage}%
\end{figure}

 Henceforth we assume there are jumps $t_0-t_1$ and $t_2-t_3$, but not
$t_4-t_5$.
 By Claim~\ref{clm:5}, at least one of $X_1$ and $X_2$ is empty. Without
loss of generality, assume $X_1= \emptyset$ and hence $t_1= u_0^+$.
 We claim that there are $K_{2,2}$ minors $M_1$ in $G[u_2,t_1]$ and
$M_2$ in $G[u_0, u_2]$, both rooted at $u_2$ and $[u_0, t_1]$.

\begin{figure}
\centering
\begin{minipage}[b]{.33\linewidth}
\centering\scalebox{.95}{\includegraphics{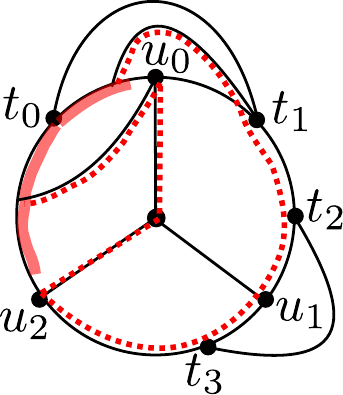}}
\captionof{figure}{\label{fig:W3}}
\end{minipage}%
\begin{minipage}[b]{.33\linewidth}
\centering\scalebox{.95}{\includegraphics{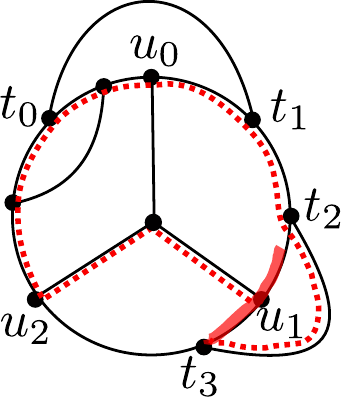}}
\captionof{figure}{\label{fig:W15}}
\end{minipage}%
\begin{minipage}[b]{.33\linewidth}
\centering\scalebox{.95}{\includegraphics{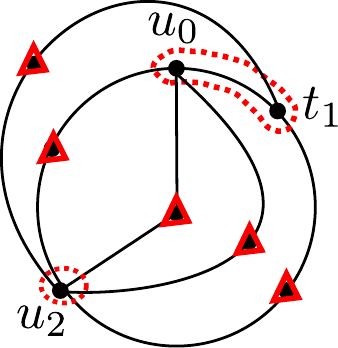}}
\captionof{figure}{\label{fig:W14}}
\end{minipage}%
\end{figure}

 Assume that $M_1$ does not exist.
 If $t_0=u_2^+$, then there is a standard longer cycle $L(u_2^+-u_0^+)$.
 Hence $t_0 \neq u_2^+$, and there must be a jump $r-r'$ from $(u_2,
t_0)$ to $(t_0,u_0]$.
 If $r' \ne u_0$ then we may take $M_1$ to be
 $\apm([u_2, r], (r',t_1]; r', t_0)$, so all jumps
from $(u_2, t_0)$ must go to $u_0$.
 If there is a $K_{2,2}$ minor along $[u_2,t_0]$ or along $[t_0, u_0]$
then we also have $M_1$, so neither of these minors exist.
 All jumps out of $(t_0,u_0)$ must go
to $t_1$ since jumps to $u_2$ are blocked by planarity. By
Corollary~\ref{cor:P1} applied to $G[u_2,t_0]$, there is a path
$P_1=t_0\ldots t$ such that $V(P_1)=(u_2,t_0]$ and $t$ is a degree $2$
vertex in $G[u_2,t_0]$, or $t=t_0$ if $(t_0,u_0)=\emptyset$, so that $t$
is adjacent to $u_0$. Similarly by Corollary~\ref{cor:P1} there is a
path $P_2=t_0\ldots s$ such that $V(P_2)= [t_0,u_0)$ and $s$ is a degree
$2$ vertex in $G[t_0,u_0]$ or $s=t_0$, so that $s$ is adjacent to $t_1$.
 Then there is a longer cycle
 $P_2 \cup st_1 \cup C[t_1, u_2] \cup u_2du_0t \cup P_1$%
 \whenfig{ as shown in Figure~\ref{fig:W3}}.
 This is a contradiction, so $M_1$ exists.

 Assume that $M_2$ does not exist.
 If there is an inside jump out of $(t_2,u_1)$ or $(u_1,t_3)$, or any
jump out of $(t_3,u_2)$, then this jump and $t_2-t_3$ give us $M_2$.
 So all edges of $G$ leaving $G[t_2, t_3]$ leave at $t_2$, $t_3$ or
$u_1$, and $(t_3,u_2) = \emptyset$.
 Any $K_{2,2}$ minor along $[t_2,t_3]$ would also provide $M_2$, so
there is no such minor.
 Therefore, by Lemma \ref{lem:P0} there is a Hamilton path $P = t_2 t_3
\ldots t$ in $G[t_2,t_3]$ with $t$ of degree $2$ in $G[t_2,t_3]$.  Then
$t=u_1$ and we have a longer cycle
 $C[u_2,t_2] \cup P \cup u_1du_2$
 \whenfig{as shown in Figure~\ref{fig:W15}}.
 This is a contradiction, so $M_2$ exists.

 Together $M_1$ and $M_2$ give a $K_{2,5}$ minor $M([u_0,t_1],
u_2)$%
 \whenfig{ as in Figure~\ref{fig:W14}}.
 This contradiction concludes the proof of Theorem~\ref{thm:main}.
 \end{proof}

 \section{Sharpness}

 A natural next step is to consider the same result for
$K_{2,6}$-minor-free graphs. It is not true, however, that all
$3$-connected planar $K_{2,6}$-minor-free graphs are Hamiltonian. In
fact, we can construct an infinite family of non-Hamiltonian
$3$-connected planar $K_{2,6}$-minor-free graphs.
 Let $G_k$ be the graph shown in Figure~\ref{fig:K2,6}, where $k \ge 1$.
 We begin by analyzing $K_{2,5}$ minors in $G_1$, which is the Herschel
graph, mentioned earlier.

 \begin{lemma}\label{lem:herschelk25}
 Suppose $G_1$ has a $K_{2,5}$ minor with standard model $\sem(R_1, R_2;
S)$.  Then

(a) $R_1 \cup R_2 \cup S = V(G_1)$,

(b) each of $R_1$ and $R_2$ contains exactly one degree $4$ vertex
of $G_1$, and

(c) $G_1$ has no edge between $R_1$ and $R_2$.

 \end{lemma}

%Figure 53
\begin{figure}
\centering
\scalebox{.7}{\includegraphics{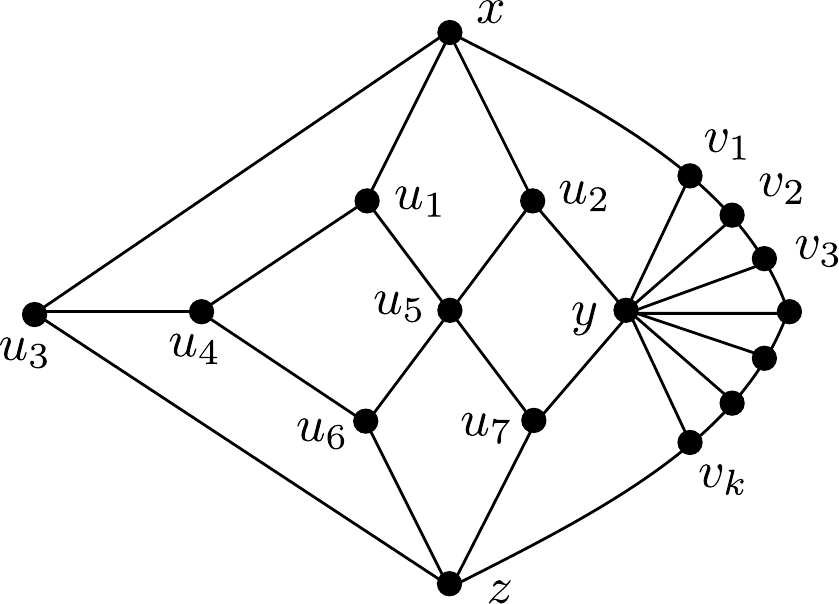}}
\caption{\label{fig:K2,6}}
\end{figure}

 \begin{proof}
 For $i = 1$ and $2$ let $H_i = G_1[R_i]$, and let $N(R_i)$ be the set of
neighbors of $R_i$ in $G_1$; then $S \subseteq N(R_i)$.
 We use the fact that $G_1$ is highly symmetric: besides the $2$-fold
symmetries generated by reflecting Figure \ref{fig:K2,6} about
horizontal and vertical axes, there is a $3$-fold symmetry generated by
the automorphism $(u_4)(u_1u_6u_3)(xu_5z)(u_2u_7v_1)(y)$.  Thus all
degree $4$ vertices in $G_1$ are similar, $u_4$ and $y$ are similar, and
all other degree $3$ vertices are similar.

 Assume without loss of generality that $|R_1| \le |R_2|$.
 Since all vertices of $G_1$ have degree $3$ or $4$, we have $|R_1| \ge
2$.
 Since $|R_1|+|R_2| \le 6$, we have $|R_1| \le 3$.
 Since $G_1$ has no triangles, $H_1$ is a path $w_1w_2$ or $w_1w_2w_3$. 
 Define the \textit{type} of a path $w_1 w_2 \ldots w_k$ to be the
sequence $d_1 d_2 \ldots d_k$ where $d_i = \deg(w_i)$.
 We break into cases according to the type of $H_1$.
 The possible types are restricted by the fact that no two degree $4$
vertices of $G_1$ are adjacent.  
 When $|R_1|=3$ we must also have $|R_2|=3$ so $H_2$ is a path
$x_1x_2x_3$, and $V(G_1)-R_1-S=R_2$ so that $V(G_1)-R_1-N(R_1) \subseteq
R_2$.

 If $H_1$ has type $33$ then $|N(R_1)| < 5$.
 If $H_1$ has type $333$ then by symmetry we may assume $H_1=u_1u_4u_6$,
and again $|N(R_1)| < 5$.  So neither of these cases happen.

 If $H_1$ has type $34$ (or $43$) then up to symmetry $H_1=u_3x$.
Then $S = N(R_1) = \set{u_1, u_2, u_4, v_1, z}$.
 Now $R_2$ contains $u_6$ (so that $u_4 \in N(R_2)$) and $y$ (so that
$v_1 \in N(R_2)$) so $H_2$ is the path $u_6u_5u_7y$ of type $3433$.

 If $H_1$ has type $334$ (or $433$) then up to symmetry $H_1=u_4u_1x$.
 Then $S=N(R_1)=\set{u_2, u_3, u_5, u_6, v_1}$ and $H_2$ is the path
$yu_7z$ of type $334$.

 If $H_1$ has type $343$ then $w_1$ and $w_3$ may be either opposite
or adjacent neighbors of $w_2$.
 If they are opposite neighbors, then up to symmetry
$H_1=u_3xu_2$.
 Then $V(G_1)-R_1-N(R_1) = \set{u_6,u_7} \subseteq R_2$ and so
either $H_2=u_6u_5u_7$ and $R_2$ is not adjacent to $v_1 \in S$, or
$H_2=u_6zu_7$ and $R_2$ is not adjacent to $u_1 \in S$.  So this does not occur.
 If $w_1$ and $w_3$ are adjacent neighbors of $w_2$, then up to symmetry
$H_1=u_1xu_2$.  Then $S=N(R_1)=\set{u_3,u_4,u_5,y,v_1}$ and $H_2$ is the
path $u_6zu_7$ of type 343.

 If $H_1$ has type $434$ then up to symmetry $H_1=xu_1u_5$.
 Then $V(G_1)-R_1-N(R_1) = \set{y,z} \subseteq R_2$ and so $H_2$ is
either $yu_7z$ or $yv_1z$.  But in either case $R_2$ is not adjacent to
$u_4 \in S$, so this case does not occur.

 Whenever the minor exists (types $34$, $334$, and $343$ with adjacent
neighbors) all of (a), (b) and (c) hold.
 \end{proof}

 \begin{proposition}
 For all $k \ge 1$, $G_k$ is a $3$-connected planar non-Hamiltonian
$K_{2,6}$-minor-free graph.
 \label{lem:K2,6}
 \end{proposition}

\begin{proof}
 In the plane embedding of $G_k$ shown in Figure~\ref{fig:K2,6} every pair
of faces intersect at most once (at a vertex or along an edge), so $G_k$
is $3$-connected.
 Let $S=\set{x, y, z, u_4, u_5}$.  Then $|S|=5$ but $G_k-S$ has six
components, so $G_k$ cannot be Hamiltonian ($G_k$ is not $1$-tough).

 We prove that $G_k$ is $K_{2,6}$-minor-free by induction on $k$.  For
$G_1$ this follows from Lemma \ref{lem:herschelk25}(a).  So suppose that
$k \ge 2$, all $G_j$ for $j \le k-1$ are $K_{2,6}$-minor-free, and
$G_k$ has a $K_{2,6}$ minor with standard model $\sem(R_1, R_2; S)$.

 Let $F = v_1v_2 \ldots v_k$.  Let $R_j'=R_j - V(F)$ for $j=1$ and $2$,
$S'=S-V(F)$, $S''=S \cap V(F)$ and $T=V(G_k)-R_1-R_2-S$.
 We cannot have $R_j \subseteq V(F)$ because any subset of $V(F)$ that
induces a connected subgraph in $G_k$ has only three neighbors in $G_k$.
 Therefore, each $R_j'$ is nonempty.
 If $v_i \in R_j \cup S$ for some $v_i \in V(F)$, then there is a path
$P_j(v_i)$ from $v_i$ to a vertex of $R_j'$, all of whose internal
vertices belong to $R_j \cap V(F)$.  The other end of $P_j(v_i)$ is one
of $x$, $y$ or $z$.

 We claim that ($\ast$) $V(F) \subseteq R_1 \cup R_2 \cup S$ and no two
consecutive vertices of $F$ belong to the same $R_j$.  If not, there is
$e \in E(F)$ with one end in $T$, or both ends in the same $R_j$. 
Contracting $e$ preserves the existence of a $K_{2,6}$ minor and gives a
graph isomorphic to $G_{k-1}$, contradicting our inductive hypothesis.

 Suppose $y \in S \cup T$.
 If some $v_a \in S$ then $P_j(v_a) = v_a v_{a-1} \ldots v_1 x$ and
$P_{3-j}(v_a) = v_a v_{a+1} \ldots z$ for $j=1$ or $2$.
 Thus $\sem(R_1', R_2'; S' \cup \set{v_1})$ is a $K_{2,6}$ minor in
$G_1$, a contradiction.
 Otherwise, by ($\ast$), $v_1 \in R_j$ and $v_2 \in R_{3-j}$ for some
$j$.
 We must have $P_j(v_1) = v_1 x$ and $P_{3-j}(v_2) = v_2v_3 \ldots v_k
z$.  Then $\sem(R_1', R_2'; S - \set{y} \cup \set{v_1})$ is a $K_{2,6}$
or $K_{2,7}$ minor in $G_1$, again a contradiction.

 So we may assume without loss of generality that $y \in R_2$.
 If $|S''| \ge 2$ we can choose $v_a, v_b \in S''$ with $a < b$ so that
there is no $v_i \in S''$ with $a < i < b$.
 Then $P_1(v_a) = v_a v_{a-1} \ldots v_1 x$ and $P_1(v_b) = v_b v_{b+1}
\ldots v_k z$, so $S'' = \set{v_a, v_b}$ and $x, z \in R_1'$.  Then
$\sem(R_1', R_2'; S' \cup \set{v_1})$ is a $K_{2,5}$ minor in $G_1$ that
contradicts Lemma \ref{lem:herschelk25}(b).
 If $|S''| \le 1$ then there is either $v_a \in S$, or since $k \ge 2$
by ($\ast$) there is $v_a \in R_1$.
 Without loss of generality $P_1(v_a)= v_a v_{a-1} \ldots v_1 x$.  Now
$\sem(R_1', R_2' \cup \set{v_1}; S')$ is a $K_{2,5}$ or $K_{2,6}$ minor
in $G_1$ with $x \in R_1'$ and $v_1 \in R_2' \cup \set{v_1}$,
contradicting Lemma \ref{lem:herschelk25}(c).
 \end{proof}

 Based on computer results of Gordon Royle (personal communication), we
suspect that it may be possible to characterize all exceptions to the
statement that all $3$-connected planar $K_{2,6}$-minor-free graphs are
Hamiltonian. All known exceptions are closely related to the family
shown in Figure~\ref{fig:K2,6}.

 \section*{Acknowledgements}

 The first author thanks Zachary Gaslowitz and Kelly O'Connell for
helpful discussions.

 \end{document}